\documentclass[12pt, reqno]{amsart}

\usepackage{Preamble/packages}
\usepackage{Preamble/environments}
\usepackage{Preamble/commands}

\begin{document}

\title{Finite group actions on the non-minimal nilpotent orbit closures in type A}

	\author[G. Bellamy]{Gwyn Bellamy}

    \author[M. Feigin]{Misha Feigin}

    \author[P. Shlykov]{Pavel Shlykov}

	\address{School of Mathematics and Statistics, University of Glasgow, University Place,
		Glasgow, G12 8QQ.}
	\email{gwyn.bellamy@glasgow.ac.uk}
        \email{misha.feigin@glasgow.ac.uk}
            \email{pavel.shlykov@glasgow.ac.uk}

	\begin{abstract}
We study quotients of the closures of the nilpotent orbits of type A by finite groups of inner automorphisms. Assuming the nilpotent orbit is not minimal, we show that the $\Q$-factorial terminalizations of quotients of nilpotent orbit closures are precisely the quotients of the symplectic resolutions of the orbit closures. This identifies the Mori fan of each $\Q$-factorial terminalization of the quotient with the Mori fan of a symplectic resolution of the orbit closure. 


	\end{abstract}

    \maketitle

\tableofcontents

  \section{Introduction} \label{sec:1.Introduction}
  Nilpotent orbit closures and their normalizations are key examples of conic symplectic singularities and play a central role in geometric representation theory, for instance in the (generalized) Springer correspondence. Beginning with Kraft--Procesi's seminal work on normality of orbit closures in classical type \cite{KPMinimal,KraftProcesi1979,KP}, a great deal of work has been done to understand the geometry \cite{FuJuteauLevySommers} of nilpotent orbit closures and their $\Q$-factorial terminalizations \cite{FuNilpotentResolutions, NamikawaInducedOrbits,FuNipQfact,NamikawaNilBirationalDuke}. Another large class of conic symplectic singularities is given by finite group quotients. One can also try to unite these two classes and consider quotients of these orbit closures by finite groups. This was done, for instance, in \cite{juteau2023minimalspecialdegenerationsduality,fu_local_geometry_special_pieces_2024} where quotients by finite groups of outer automorphisms are considered in order to prove Lusztig's conjecture on special pieces in exceptional Lie algebras. In this paper, we consider instead quotients by finite groups of inner automorphisms. We concentrate on the simplest non-trivial case, that of nilpotent orbit closures of type $\mathsf{A}$. Already this case is quite involved. 

\subsection*{Main result}

Recall that the nilpotent $G := GL(n)$-orbits $\mc{O}$ in $\mf{sl}_n$ are labelled by partitions of $n$, corresponding to the Jordan type of the matrices in the orbit. 
Their closures $\overline{\mc{O}}$ are normal and have symplectic singularities. 
We pick a finite group $\Gamma \subset G$ that intersects the centre of $G$ trivially and consider the quotient $\overline{\mc{O}}/\Gamma$. This is again a conic symplectic singularity. 
As shown by Fu \cite{FuNilpotentResolutions}, all projective symplectic resolutions of $\overline{\mc{O}}$ are of the form $\rho \colon T^* (G/P) \to \overline{\mc{O}}$ for an appropriate choice of parabolic subgroup $P \subset G$. 
Since $G$ acts on $T^* (G/P)$ such that the resolution $\rho$ is $G$-equivariant, we can consider the induced morphism $(T^* (G/P)) / \Gamma \to \overline{\mc{O}}/\Gamma$. Our main result is the following:

\begin{theorem}\label{thm:mainintro}
    Assume $\mc{O}$ is not the minimal nilpotent orbit $\mc{O}_{\mr{min}}$ (and not the zero orbit). Then taking quotients
    \[
    ( X \to \overline{\mc{O}} ) \mapsto (X / \Gamma \to \overline{\mc{O}}/\Gamma )
    \]
    defines a bijection between the (finitely many) isomorphism classes of projective symplectic resolutions of $\overline{\mc{O}}$ and the isomorphism classes of $\Q$-factorial terminalizations of $\overline{\mc{O}}/\Gamma$. 
\end{theorem}
Here, two resolutions (respectively, $\Q$-factorial terminalisations) $X \to Y, X' \to Y$ of $Y$ are isomorphic if there is an isomorphism $X \iso X'$ commuting with the morphisms to the base $Y$. 
In particular, \cref{thm:mainintro} implies that $\overline{\mc{O}}/\Gamma$ does not admit a projective symplectic resolution when the group $\Gamma$ is non-trivial, as $X/\Gamma$ is not smooth in this case, as we explain before \cref{thm:non-minimal_quotients_are_q-factorial_and_terminal}. Moreover, we show in \cref{cor:pistarmorifannip} that if $X = T^* (G/P)$ resolves $\overline{\mc{O}}$ and $\pi \colon X \to X / \Gamma$ is the quotient morphism then there is an isomorphism 
\[
\pi^* \colon \mr{Pic}(X/\Gamma) \otimes_{\Z} \Q \iso \mr{Pic}(X) \otimes_{\Z} \Q, 
\]
which identifies the Mori fan of $X/\Gamma$ with the Mori fan of $X$.  In particular, as in Remark~\ref{rem:Morifan}, this implies that every crepant projective partial resolution of $\overline{\mc{O}}/\Gamma$ is of the form $X / \Gamma \to \overline{\mc{O}}/\Gamma$ for a unique crepant partial resolution $X \to \overline{\mc{O}}$.

As we show in our companion paper \cite{BFS_minimal_nilpotent}, Theorem~\ref{thm:mainintro} does not necessarily hold for the minimal nilpotent orbit $\mc{O}_{\mr{min}}$. We consider in detail the quotient $\overline{\mc{O}}_{\mr{min}}/\s_n$, where $\s_n$ is realised as permutation matrices in $G$. In this case, the projective symplectic resolutions are of the form $T^* \mathbb{P}^{n-1} \to \overline{\mc{O}}_{\mr{min}}$ and the crepant partial resolutions $(T^* \mathbb{P}^{n-1}) / \s_n$ do not have terminal singularities. Moreover, $\overline{\mc{O}}_{\mr{min}}/\s_n$ has many more $\Q$-factorial terminalizations than $\overline{\mc{O}}_{\mr{min}}$ has projective symplectic resolutions (there are only two of the latter, barring the case $n=2$ which we consider separately).  In \cite{BFS_minimal_nilpotent} we construct all $\Q$-factorial terminalizations of $\overline{\mc{O}}_{\mr{min}}/\s_n$ via variation of Geometric Invariant Theory (GIT). 

One can ask to what extent our results generalize to nilpotent orbit closures of other types. The theory we develop is completely general. But we require that the complement to the locus in $G/P$ where $\Gamma$ acts freely has codimension at least two. We show that this is true for all nontrivial parabolics in type $\mathsf{A}$ \textit{except} for the two maximal parabolics giving rise to $\mathbb{P}^{n-1}$. We have not analysed the extent to which it is also true in other types -- where it also holds, then our results will also be true.

\subsection*{AI declaration}
 ChatGPT $5.6$ Sol and $6.0$ Astra were used for proofreading, as well as for helping us locate several citations. The proofs of \cref{lem:Qfactorialterminaliffsmallmodification}, \cref{lem:functorial-resolution-codim-two} and \cref{prop:Gactionlifts} were proposed by ChatGPT 
 $5.6$ Sol. The authors bear full responsibility for the contents of this paper.

\subsection*{Acknowledgements}

The first author was partially supported by a Research Project Grant from the Leverhulme Trust and by the Engineering and Physical Sciences Research Council [grant numbers EP/W013053/1, EP/R034826/1]. 
The second and third authors were supported by the Engineering and Physical Sciences Research Council [grant number EP/W013053/1].


  \section{Preliminaries} \label{sec:2.Preliminaries}

Denote by $\Gm$  the algebraic group $\Cs$. 
If $H$ is an abelian group then $H_{\Q} := H \otimes_{\Z} \Q$. If $G$ is a group then $\mathbb{X}^{\bullet}(G)$ is the group of linear characters $G \to \Gm$.
For a group $G$ acting on a variety $M$ we denote by $M^g$ the set of points in $M$  fixed under the action of an element $g \in G$. 
\subsection{Algebraic geometry definitions}

Throughout we consider separated schemes of finite type over $\C$. A \textit{variety} will mean a reduced and irreducible scheme. Moreover, we assume throughout that our varieties are quasi-projective and have canonical singularities. In particular, if $\Gamma$ is a finite group acting on a variety $X$ then this assumption ensures that $X$ admits a cover by $\Gamma$-stable affine open subsets and hence the quotient $X/\Gamma$ exists (and is again quasi-projective) \cite[Lemma~1.3]{BertinNotes}. Let $K(X)$ be the field of rational functions on  $X$ and denote by $K(X)^{\times}$ the set $K(X)\setminus \{0\}$. Denote by $\mathcal{O}_X$ the sheaf of regular functions on $X$, and let $\mathcal{O}_X^\times$ be the subsheaf of units in $\mathcal{O}_X$. 

Assume that $X$ is a normal variety. We recall that $\mathrm{Div}(X)$ denotes the free $\Z$-module generated by all irreducible subvarieties (prime Weil divisors) of codimension one in $X$ and $\mathrm{Cl}(X)$ is the quotient of $\mathrm{Div}(X)$ by the subgroup generated by all principal (Weil) divisors. We denote by $\mathrm{Pic}(X)$ the Picard group of $X$, whose elements are isomorphism classes of line bundles and addition is given by tensor product of line bundles. The following is standard:

\begin{lemma}\label{lem:PicClnormalembedding}
    If $X$ is a normal variety then the map $\mathrm{Pic} (X) \rightarrow \mathrm{Cl}(X)$ is injective. 
\end{lemma}


\begin{defn} \label{defn:cartier}
A Cartier divisor on $X$ is an equivalence class of a collection $\{ (U_i, f_i)  \}$, where $\{ U_i \}$ is an open cover of $X$ and the $f_i$ are sections of $K(X)^{\times}$ on $U_i$ with $f_i f_j^{-1} \in \mc{O}_X(U_i \cap U_j)^{\times}$. 
\end{defn}

The group $\mathrm{CaCl}(X)$ of classes of Cartier divisors modulo global principal divisors can be identified with $\mathrm{Pic} (X)$ and hence is a subgroup of $\mathrm{Cl}(X)$ when $X$ is normal. 
A Weil divisor $D$ is called $\Q$\textit{-Cartier} if 
 there exists a positive integer $n$ such that $nD$ is a Cartier divisor. 

    \begin{defn} \label{defn:q-factorial}
        A normal variety $X$ is \textit{$\Q$-factorial} if every Weil divisor $D \in \mathrm{Div}(X)$ is $
        \Q$-Cartier.
    \end{defn}

\subsection{Line bundles on quotient varieties}
Denote by $\mathrm{Pic}(A)$ the Picard group of the ring $A$, the group of isomorphism classes of invertible $A$-modules.
\begin{lemma}\label{lem:triviallinebundlesoncone}
	Let $A$ be an $\N$-graded affine, integrally closed $\C$-algebra with $A_0 = \C$. Then $\mathrm{Pic}(A) = \{ 0 \}$. 
\end{lemma}

\begin{proof}
	We follow the proof given in Lemma 5.1 of \cite{Murthy}. Let $\mf{m} = A_+$ be the augmentation ideal.
Since $A$ is normal, we have a commutative diagram 
	$$
	\begin{tikzcd}
	\mathrm{Pic}(A) \ar[r,hook] \ar[d] & \mathrm{Cl}(A) \ar[d] \\
	\mathrm{Pic}(A_{\mf{m}}) \ar[r,hook] & \mathrm{Cl}(A_{\mf{m}})  
	\end{tikzcd}
	$$
	where the horizontal arrows are inclusions by \cref{lem:PicClnormalembedding}. 

    We recall that Proposition 7.4 from \cite{Samuel} says the following. Let $A$ be an $\N$-graded Krull ring with $A_0 = \C$ a field. Let $\mf{m} = A_+$ be the augmentation ideal. Then the canonical map $\mathrm{Cl}(A) \rightarrow \mathrm{Cl}(A_{\mf{m}})$ is an isomorphism. 

    It follows that the vertical map on the right is an isomorphism. Therefore the left vertical arrow is an inclusion. But $\mathrm{Pic}(A_{\mf{m}}) = \{ 0 \}$ since every finitely generated projective module over a local ring is free. 
\end{proof}
	
Geometrically,  \cref{lem:triviallinebundlesoncone} says that if $X$ is a normal affine cone of finite type over $\C$ then $\mathrm{Pic}(X) = 0$. 




\begin{lemma} \label{lem:ker_of_factor_map_for_CL}
    Let $X$ be a normal affine variety on which a finite group $\Gamma$ acts. Then pullback along $\pi \colon X \to X / \Gamma$ induces an isomorphism $\pi^* \colon \mr{Cl}(X/\Gamma)_{\Q} \to \mr{Cl}(X)^{\Gamma}_{\Q}$.  
\end{lemma}

\begin{proof}
This follows from \cite[Example~1.7.6]{FultonIntersection} by considering the degree $\dim X - 1$ part of the rational Chow group. 

\end{proof}



\begin{lemma}\label{lem:pullbackgroupiso}
    Let $X$ be a variety and $\Gamma$ a finite group acting on $X$.
    \begin{enumerate}
        \item[(i)] Pullback induces an isomorphism $\pi^* \colon \Pic(X/\Gamma)_{\Q} \iso \Pic(X)_{\Q}^{\Gamma}$ of rational vector spaces. 
        \item[(ii)] If $\Gamma$ acts trivially on $\mathrm{Pic}(X)$ then $\pi^* \colon \Pic(X/\Gamma)_{\Q} \iso \Pic(X)_{\Q}$. 
        \item[(iii)] If $X$ is normal and the action of $\Gamma$ is the restriction of the action of a connected linear algebraic group $G$ then $\pi^* \colon \Pic(X/\Gamma)_{\Q} \iso \Pic(X)_{\Q}$. 
    \end{enumerate}
\end{lemma}

\begin{proof} 
Part (i). We first note that by \cite[0BCY]{stacks} $\pi^* \colon \Pic(X/\Gamma)_{\Q} \to \Pic(X)_{\Q}$ is injective, with image contained in $\Pic(X)_{\Q}^{\Gamma}$. 
Therefore, it suffices to show that $\pi^*$ is surjective after tensoring by $\Q$. Let $L$ be a line bundle on $X$ whose isomorphism class is fixed by $\Gamma$. Then, by \cite[Proposition 3.4.5]{Brion-linearization} there exists $m > 0$ such that $L^{\otimes m}$ can be endowed with the structure of a $\Gamma$-equivariant line bundle. 
Moreover, since the stabilizers of points in $X$ are all finite, replacing $m$ by a multiple if necessary, we may assume that $\Gamma_x$ acts trivially. Then it follows from descent \cite[Proposition~3]{KraftGbundles} that there is a line bundle $M$ on $X/\Gamma$ such that $L^{\otimes m} \cong \pi^* M$.  

Part (ii) follows directly from (i). 

Part (iii). We need to show that $\pi^*$ is surjective. Let $L \in \Pic(X)$. Since $G$ is connected and $X$ is normal,  \cite[Proposition~2.4]{KKLV} says that $L^{\otimes n}$ is $G$-linearizable for some $n > 0$. In particular, $L^{\otimes n}$ is $\Gamma$-linearizable and hence $\Gamma$ acts trivially on the isomorphism class of $L^{\otimes n}$. By part (i), there exists $m > 0$ and $M \in \Pic(X/\Gamma)$ such that $\pi^* M \cong (L^{\otimes n})^{\otimes m}$. Hence $[L] = \frac{1}{mn} \pi^* [M]$ in $\Pic(X)_{\Q}$.  
\end{proof}

The following was noted without proof in \cite[\S~2.4]{BellSchedler3large}. 

\begin{lemma}\label{lem:Qfactorialquotient}
    If $X$ is a $\Q$-factorial normal variety and $\Gamma$ a finite group acting on $X$ then $X/\Gamma$ is $\Q$-factorial and normal. 
\end{lemma}

\begin{proof}

From now on, denote by $\pi$ the morphism $X \rightarrow X/\Gamma$. 
Being $\Q$-factorial is a Zariski local property. Moreover, since all varieties (in particular $X$) are assumed to be quasi-projective, we may cover $X$ by $\Gamma$-stable affine open subsets. 
Therefore, considering each affine open set in turn, we may assume that $X$ is affine. 
This gives rise to the diagram 
\begin{equation}
\begin{tikzcd}
    \mathrm{Pic}(X)^{\Gamma} \ar[r,hook,"d_X"] & \mathrm{Cl}(X)^{\Gamma} \\
    \mathrm{Pic}(X/\Gamma) \ar[u,"\pi^*"] \ar[r,hook,"d_{X/\Gamma}"] & \mathrm{Cl}(X / \Gamma) \ar[u,"\pi^*"']
\end{tikzcd}
\label[diagram]{diag:Qfactorialquotient}
\end{equation}
where the pullback $\pi^*$ on the right is the map in \cref{lem:ker_of_factor_map_for_CL}. We claim that this diagram is commutative. If we identify line bundles with Cartier divisors then $\pi^* \{ (U_i,f_i) \} = \{ (\pi^{-1}(U_i),f_i) \}$, where $f_i \in \mathrm{Frac} \, \mc{O}(X)^{\Gamma} = (\mathrm{Frac} \, \mc{O}(X))^{\Gamma} \subset \mathrm{Frac} \, \mc{O}(X)$. Then $d_{X/\Gamma}( \{ (U_i,f_i) \} ) = \sum_{\mf{p}} v_{\mf{p}}(f_i) [\mf{p}]$, where the sum is over all prime ideals of height one in $\mc{O}(X)^{\Gamma}$, $[\mf{p}]$ is the class of $\mf{p}$ in $ \mathrm{Cl}(X / \Gamma)$ and the expression is independent of the choice of $i$; see \cite[Proposition~6.11]{Hartshorne}. The ramification index $e(\mf{B},\mf{p})$ is defined so that $v_{\mf{B}} = e(\mf{B},\mf{p}) v_{\mf{p}}$ if $\mf{B} \cap \mc{O}(X)^{\Gamma} = \mf{p}$. If $\pi^*(\mf{p}) = \sum_{\mf{B}} e(\mf{B},\mf{p}) \mf{B}$ then 
\begin{align*}
\pi^*(d_{X/\Gamma}(\{ (U_i,f_i) \})) & = \pi^* \left(\sum_{\mf{p}} v_{\mf{p}}(f_i) [\mf{p}] \right) \\
& = \sum_{\mf{p}} \sum_{\mf{B}} v_{\mf{p}}(f_i) e(\mf{B},\mf{p}) [\mf{B}] \\
& = \sum_{\mf{B}} v_{\mf{B}}(f_i) [\mf{B}] = d_X(\pi^* \{ (U_i,f_i) \} ),
\end{align*}
as claimed. By \cref{lem:ker_of_factor_map_for_CL}, $\pi^*$ is an isomorphism. Together with \cref{lem:pullbackgroupiso}(i), this implies that, after tensoring with $\Q$, the vertical maps are isomorphisms. Since $X$ is assumed to be $\Q$-factorial, the map $\mathrm{Pic}(X)_{\Q} \to \mathrm{Cl}(X)_{\Q}$ is an isomorphism. Since $\Gamma$ is finite, taking invariants is exact over $\Q$ and gives
\[
\mathrm{Pic}(X)^{\Gamma} \otimes_{\Z} \Q = (\mathrm{Pic}(X)_{\Q})^{\Gamma} \iso (\mathrm{Cl}(X)_{\Q})^{\Gamma} = \mathrm{Cl}(X)^{\Gamma} \otimes_{\Z} \Q. 
\]
Thus, the top horizontal map is also an isomorphism over $\Q$. We deduce that the bottom horizontal map is an isomorphism over $\Q$. Lastly, normality is preserved under quotients by finite groups.
\end{proof}

\begin{example}
    The converse to \cref{lem:Qfactorialquotient} is false. For instance, if $X = \overline{\mathcal{O}}_{\mathrm{min}} \subset \lie{sl}_n$ for $n>2$ then matrix transposition $(-)^t$ is an involution on $X$. Thus, there is an action of $\Z_2$ on $X$. One can show that this acts non-trivially on $\mathrm{Cl}(X) = \Z$. Therefore, $\mathrm{Cl}(X)^{\Z_2} = 0$ and hence \cref{lem:ker_of_factor_map_for_CL} implies that the quotient $\overline{\mathcal{O}}_{\mathrm{min}} / \Z_2$ has finite class group. In particular, it is $\Q$-factorial. However, $\overline{\mathcal{O}}_{\mathrm{min}}$ is not $\Q$-factorial since its Picard group is trivial by \cref{lem:triviallinebundlesoncone}.
\end{example}



The following is presumably well-known but we do not know of a reference.

\begin{lemma}\label{lem:grouprojquotient}
    Let $A = \oplus_{n \ge 0} A_n$ be a graded affine $\C$-algebra and $\Gamma$ a finite group acting on $A$ by graded algebra automorphisms. Let $X = \Proj A$ and $Y = \Proj A^{\Gamma}$. Then the morphism $\pi \colon X \to Y$ induced by the inclusion $A^{\Gamma} \hookrightarrow A$ is the categorical quotient. 
\end{lemma}

\begin{proof}
    To show that $\pi$ is a categorical quotient it suffices to exhibit an affine open cover $\{ V_i \}$ of $Y$ such that $\{ U_i = \pi^{-1}(V_i) \}$ is an affine ${\Gamma}$-stable covering of $X$ with $\mc{O}(V_i) = \mc{O}(U_i)^{\Gamma}$ for all $i$. In fact, we construct the $U_i$ first as $U_i = U_{f_i}$, where $U_{f_i} = \Spec A_{(f_i)}$ and $f_i \in A^{\Gamma}_{> 0}$ is homogeneous. If we take $V_i = V_{f_i}$ similarly, then it is clear that 
    \[
    \mc{O}(V_i) := (A^{\Gamma})_{(f_i)} = (A_{(f_i)})^{\Gamma} = \mc{O}(U_i)^{\Gamma}.
    \]
    We just need to show that there exist $f_1, \dots, f_k \in A_{> 0}^{\Gamma}$ such that the $U_{f_i}$ cover $X$ and the $V_{f_i}$ cover $Y$. Take the $f_i$ to be homogeneous generators of $A_+^{\Gamma}$ over $A^{\Gamma}$. Then the latter claim is clear. Let $I = A \cdot A_{> 0}^{\Gamma}$ be the ideal in $A$ generated by the $f_i$. We need to show that if $\mf{p} \lhd A$ is a graded prime ideal containing $I$ then $\mf{p}$ contains $A_+$ (and thus $\mf{p} \notin \Proj A$). It suffices to show that there exists some $N > 0$ such that $A_+^N \subset I$. Now $A$ is a finite graded $A^{\Gamma}$-module, therefore so too is $A_+$. This means we can choose homogeneous elements $h_1, \dots, h_r \in A_+$ such that $A_+ = A^{\Gamma} \cdot \{ h_1, \dots, h_r \}$. The images $\bar{h}_i$ of the $h_i$ in $A_+/ I$ form a spanning set as a $(A^{\Gamma} / A_+^{\Gamma})\cong A_0^{\Gamma}$-module. Take any homogeneous $h \in A_+ / I$ and note that $\deg h > 0$. Then we may write 
    \[
    h^N = \sum_{i} a_{i,N} \bar{h}_{i},
    \]
    where $a_{i,N} \in A_0^{\Gamma}$ and $a_{i,N} = 0$ unless $\deg \bar{h}_i = \deg h^N$. But this implies that $h^N = 0$ if $N > \max \{ \deg h_i \}$. Thus, $A_+^N \subset I$ as required. 
\end{proof}

\begin{remark}
    Note that the key point in the proof of \cref{lem:grouprojquotient} is that $\sqrt{A_+^{{}}} = \sqrt{A \cdot A_+^{\Gamma}}$.  
\end{remark}

\begin{lemma}\label{lem:finitequotientsing}
    Suppose a finite group $\Gamma$ acts symplectically on a smooth symplectic variety $M$ and that this action is faithful. Let $\pi \colon M \to M / \Gamma$ be the quotient morphism. A point $\pi(x) \in M/ \Gamma $  is singular if and only if $x \in M$ has non-trivial stabilizer. 
\end{lemma}

\begin{proof}
Let $g \in \Gamma$ and assume that $g \cdot m = m$. Then $g$ acts symplectically on $T_m M$. This means that the set of fixed points $(T_m M)^g$ is a symplectic subspace of $T_m M$ and hence has even codimension. By Cartan's Lemma, $M^g$ is smooth and $(T_m M)^g = T_m M^{g}$. Thus if $(T_m M)^g = T_m M$ then $\dim M^g = \dim M$ which implies that $M^g = M$ and hence $g = 1$ since the action is assumed faithful. Therefore the codimension of $(T_m M)^g$ in $T_m M$ is at least two. This implies that the codimension of $M^g$ in $M$ is at least two. 

    \begin{claim}
Suppose $M$ is a smooth variety with a faithful action of a finite group $\Gamma$ such that $\mathrm{codim}_M M^g \ge 2$ for all $g \in \Gamma \setminus \{ 1 \}$. Let $Z = \left \{ x \in M \, | \,  \on{Stab}_{\Gamma}(x) \neq \{ 1 \} \right \}$. Then 
\[
\on{Sing}(M/\Gamma)= \pi( Z).
\]
\end{claim}

\begin{proof}
    The claim follows from Cartan's Lemma, see \cite[Lemma 2.8]{Etingof_finite_group_action}, which says that in a formal neighbourhood of each $m \in M$ the action can be linearized, together with the Chevalley–Shephard–Todd Theorem.
\end{proof}
 
The claim applies and the lemma follows. 
\end{proof}

  \section{Geometric invariant theory} \label{sec:3.GIT}
  
\subsection{Geometric invariant theory}

In this section, we provide a reminder of the basic facts about GIT quotients for affine varieties that are required later. 

 Suppose $G$ is a  connected reductive group acting on an affine variety $Y = \Spec A$. The affine GIT quotient is the variety $Y \git G:= \Spec A^G$, where $A^G$ is the ring of $G$-invariant elements in $A$. Fix a character $\chi \in \ch(G) $. For $n \ge 0$, the space of semi-invariants of weight $n\chi$ is then given by
 \begin{equation*}
  A_{n\chi}= \{ f \in A \, | \,  f(g^{-1} \cdot y)= \chi(g)^nf(y) \text{ for all } g \in G, y \in Y \} . 
 \end{equation*}
 The \textit{semi-stable locus} of $\chi$ in $Y$ is the open set     
 \begin{equation*}
     Y^{\chi} := \{ y\in Y \, | \,  \text{ there exists } n>0, f\in A_{n\chi} \text{ such that } f(y) \neq 0    \}.
 \end{equation*}
 We define the GIT quotient with respect to $\chi$ to be 
 \begin{equation*}
     Y \git_{\chi} \, G:=  \Proj \left ( \bigoplus_{n \geq 0} A_{n\chi} \right ).
 \end{equation*}
The semi-stable locus varies as we vary the character $\chi$. The space  $\ch(G)_\Q$ can be decomposed into regions, under the equivalence relation 
\begin{equation*}
    \chi, \chi' \text{ belong to the same region} \Leftrightarrow \text{the sets } Y^{\chi} \text{ and } Y^{\chi'} \text{ coincide.}  
\end{equation*}
 The part of this decomposition, supported on the (positive multiples of) characters with non-empty semistable locus is called the \textit{GIT fan} $\Sigma(Y)$ of the action of $G$ on $Y$. It is shown, for instance in \cite[Theorem 3.2]{ARZHANTSEV-git_via_cox_rings}, that it is given by a finite fan consisting of rational polyhedral cones. The interiors of the cones of top dimension are called \textit{chambers}.

\subsection{GIT fan of a quotient variety}
 
Let $Y$ be an affine variety, $G$ a connected reductive group and $\Gamma$ a finite group. Assume that $G \times \Gamma$ acts on $Y$. In this section, we prove the following result. 

\begin{proposition}\label{prop:GITfansthesame}
The GIT fan $\Sigma(Y)$ for $G$ acting on $Y$ coincides with the GIT fan $\Sigma(Y/\Gamma)$ for $G$ acting on $Y/\Gamma$.
\end{proposition}

\begin{proof}
    The statement of the proposition is equivalent to the following. If $\theta,\theta'$ are characters of $G$ then $Y^{\theta} = Y^{\theta'}$ if and only if $(Y / \Gamma)^{\theta} = (Y / \Gamma)^{\theta'}$. Since the set $Y^{\theta}$ is $\Gamma$-stable, $Y^{\theta} = \pi^{-1}(Y^{\theta} / \Gamma)$, where $\pi \colon Y \to Y/\Gamma$ is the quotient morphism. Therefore, the required statement is a consequence of \cref{lem:GITtheWTequal} below. 
\end{proof}

\begin{lemma}\label{lem:GITtheWTequal}
For any linear character $\theta$ of $G$, $Y^{\theta} / \Gamma \;=\; (Y/\Gamma)^{\theta}$.
\end{lemma}

\begin{proof}
First, we may choose a linear representation $U$ of $G \times \Gamma$ such that $Y$ embeds equivariantly as a closed subvariety of $U$.  
Then
\[
  Y^{\theta} = U^{\theta} \cap Y, \qquad 
  (Y/\Gamma)^{\theta} = (U/\Gamma)^{\theta} \cap (Y/\Gamma), \qquad 
  Y^{\theta}/\Gamma = (U^{\theta}/\Gamma) \cap (Y/\Gamma).
\]
Therefore, it suffices to assume that $U = Y$.  Let $\mathbb{Y}(G)$ denote the space of one-parameter subgroups of $G$. If $\chi_n \colon \Gm \to \Gm$ is the character given by $\chi_n(t) = t^n$ then there is a pairing $\langle - , - \rangle \colon \mathbb{Y}(G) \times \mathbb{X}^{\bullet}(G) \to \Z$ defined by $\langle \lambda, \theta \rangle = n$ if $\theta \circ \lambda = \chi_n$. A one-parameter subgroup $\lambda \in \mathbb{Y}(G)$ acts on $U$ by the composition $\Gm \stackrel{\lambda}{\rightarrow} G \to \mathrm{Aut}(U)$. Then we write $U^{\theta,\lambda}$ for the set of $\chi_n$-semistable points with respect to this action of $\Gm$, where $n = \langle \lambda, \theta \rangle$. We claim that 
\[
U^{\theta} = \bigcap_{\lambda \in \mathbb{Y}(G)} U^{\theta,\lambda}.
\]
It is clear that if $u$ is $\theta$-semistable for the $G$-action then it is also $\chi_{n}$-semistable (for $n = \langle \lambda, \theta \rangle$) with respect to the $\Gm$-action defined by $\lambda$. Thus, $U^{\theta} \subset \bigcap_{\lambda \in \mathbb{Y}(G)} U^{\theta,\lambda}$. For the converse, we apply the Hilbert–Mumford criterion, as stated in \cite[Proposition~2.5]{KingStable}. If $u \in U \setminus U^{\theta}$ then there exists a one-parameter subgroup 
$\lambda \colon \Gm \to G$ such that the limit 
\[
  \lim_{t \to 0} \lambda(t) \cdot u
\]
exists and $\langle \lambda, \theta \rangle > 0$. But this means that, under the action of $\Gm$ on $U$ via $\lambda$: 
\begin{enumerate}
    \item $U = \bigoplus_{n \in \Z} U_n$ with $\lambda(t) \cdot v = t^n v$ for $v \in U_n$ and $t \in \Gm$ 
    \item $U \setminus U^{\theta,\lambda} = \bigoplus_{n \ge 0} U_n$.
\end{enumerate} 
Point (2) implies that $u \in U \setminus U^{\theta,\lambda}$. Hence $U^{\theta} \supset \bigcap_{\lambda \in \mathbb{Y}(G)} U^{\theta,\lambda}$.

Choosing an embedding of $U / \Gamma$ into some linear representation of $G$, the same argument shows that $(U/\Gamma)^{\theta} = \bigcap_{\lambda \in \mathbb{Y}(G)} (U/\Gamma)^{\theta,\lambda}$. Hence the equality $U^{\theta} / \Gamma = (U/\Gamma)^{\theta}$ follows from \cref{lem:HMCsWact} below.
\end{proof}

\begin{lemma}\label{lem:HMCsWact}
Let $U$ be a finite-dimensional linear representation of $\Gm \times \Gamma$ and let $\theta  \in \Z_{> 0}$, thought of as a character of $\Gm$. Then $U^{\theta} / \Gamma \;=\; (U/\Gamma)^{\theta}$. 
\end{lemma}

\begin{proof}
We decompose $U = \bigoplus_{n \in \mathbb{Z}} U_{n}$ as a $\Gm$-module, where each $U_{n}$ is a $\Gamma$-module. Write $\mathcal{O}(U) = \C[x_1, \ds, x_k]$, where each $x_i$ is an eigenvector for $\Gm$, $k = \dim U$. Then $U \setminus U^{\theta} = Z$, where 
\[
  Z = V(x_{i} \mid \deg(x_{i}) > 0) = \bigoplus_{n \geq 0} U_{n}.
\]
We write $U = Z \oplus U^{-}$, where $U^{-} = \bigoplus_{n < 0} U_{n}$.  

Now
\[
  (U/\Gamma) \setminus (U^{\theta}/\Gamma) = Z/\Gamma = V(I^{\Gamma}),
\]
where $I$ is the ideal in $\mathbb{C}[x_{1},\dots,x_{k}]$ generated by the set $\{x_{i} \mid \deg(x_{i}) > 0\}$.  

On the other hand,
\[
  (U/\Gamma) \setminus (U/\Gamma)^{\theta} = V(J),
\]
where $J$ is the ideal generated by the space $\mathbb{C}[U]^{\Gamma}_{>0}$ of $\Gamma$-invariant functions of positive weight with respect to $\Gm$. Since $J \subset I^{\Gamma}$, we have $Z/\Gamma \subset V(J)$, or equivalently, $(U/\Gamma)^{\theta} \subset U^{\theta}/\Gamma$. 

It remains to prove the converse inclusion. Let $[u] \in U^{\theta}/\Gamma$ with $u \in U^{\theta}$. We decompose $u = z + u^{-}$ with $z \in Z$, $u^{-} \in U^{-}$.  
By assumption, $u^{-} \neq 0$.  

The fibre over $0$ of the morphism $U^{-} \to U^{-}/\Gamma$ is just the origin $\{ 0 \}$. Therefore, there exists $f \in \mathbb{C}[U^{-}]^{\Gamma}_{+}$ such that $f(u^{-}) \neq 0$.  
We may assume that $f$ is homogeneous with respect to the $\Gm$-action.  
Thus $f \in \mathbb{C}[U]^{\Gamma}_{>0}$ and $f(u) = f(u^{-}) \neq 0$. Hence $[u] \in (U/\Gamma)^{\theta}$.  
This proves the desired equality.
\end{proof}

  \section{Birational geometry and symplectic singularities} \label{sec:4.Birgeom}
  \subsection{Birational geometry and Mori dream spaces}\label{sec:birationaldefn}

In this section, we recall the basic definitions and results from the minimal model program that we require. Since we wish to describe certain partial resolutions of an affine symplectic singularity, we work in the relative setting, as in \cite{KawamataMatsudaMatsuki}. In particular, we wish to classify all \textit{$\Q$-factorial terminalizations} of our singularity.
The fact that these $\Q$-factorial terminalizations exist (and are finite) is the famous result \cite{BCHM}. However, we can construct them directly in our case and do not require this abstract existence result.  

Let $Y$ be a normal variety such that its canonical class $K_Y$ is $\Q$-Cartier. Let $\rho: X \rightarrow Y$ be a resolution of singularities. If we denote by $E_i$ the prime $\rho$-exceptional divisors, then
\begin{equation}   
\label{terminal_resolution}
K_X= \rho^*K_Y+\sum a(E_i,Y)E_i, \quad \textrm{for some $a(E_i,Y) \in \Q$,}
\end{equation}
in $\mathrm{Cl}(X)$. The coefficients $a(E_i,Y)$ are called \textit{discrepancies} of $E_i$ with respect to $Y$.    Moreover, if all $a(E_i,Y)=0$ then such a resolution is called \textit{crepant}. 

\begin{defn}\label{defn:do-not-use-terminal}
A normal variety $X$ is said to have terminal singularities if there exists a resolution $\rho$ such that all the coefficients $a_i$ in \eqref{terminal_resolution} are positive.  
\end{defn}

\begin{defn} 
Let $X$ and $Y$ be normal quasi-projective varieties and $\rho \colon X \to Y$ a projective morphism. 
    \begin{enumerate}[label=(\roman*), ref=\thedefn(\roman*)]
        \item \label[Definition]{defn:algfibrespace} 
        We say that $\rho$ is an \emph{algebraic fiber space} if $\rho_{*}\mathcal{O}_{X} = \mathcal{O}_{Y}$.
        \item \label[Definition]{dfn:crepant-partial} Suppose in addition to (i) that the canonical divisor $K_Y$ is $\Q$-Cartier, $\rho$ is crepant and an isomorphism over the regular locus. 
        We then call $\rho$ \textit{crepant projective partial resolution.}
        \item \label[Definition]{dfn:q-factorial_terminalisation} Suppose also that $X$ is $\Q$-factorial and has terminal singularities. 
        Then $\rho$ is called a $\Q$\textit{-factorial terminalisation.}
    \end{enumerate}
\end{defn}
Throughout the remainder of \cref{sec:birationaldefn}, we assume that $\rho \colon X \to Y$ is a fixed algebraic fiber space. 
We will also assume our (partial) resolutions to always be projective and omit the word projective starting from now.

\begin{defn} \label{defn:picard_united_general_notation}
    We will be using the following objects and notation for them in this paper:
    \begin{enumerate}
        \item We mainly consider the vector space $\mathrm{Pic}(X)_{\Q}$. An element of $\mathrm{Pic}(X)_{\Q}$ will be called a \textit{fractional line bundle}. Note that some multiple of a fractional line bundle is represented by an actual line bundle.  
        \item Assume now that $\rho \colon X \to Y$ is a projective morphism between normal varieties. The \textit{relative Picard group} is $\Pic(X/Y ) := \Pic(X)/ \rho^* \Pic(Y )$.
        \item Recall that a line bundle $\mc{L}$ is \textit{ample} if  $\mc{L}^{\otimes m}$ is \textit{very ample} for some $m\in \N$, i.e., its global sections give an embedding to a projective space. Now a line bundle $\mc{L}$ on $X$ is $\rho$\textit{-ample} if for every affine open subset $U \subset Y$ the restriction $\mc{L}_{\rho^{-1}(U)}$ is an ample line bundle on $\rho^{-1}(U)$.
        \item A line bundle $\mc{L} \in \Pic(X/Y )_{\Q}$ is $\rho$-\textit{nef} if $\deg \mc{L}|_{\ell} \ge 0$ for every (proper) curve $\ell$ contracted by $\rho$.  
        \item We say that a (fractional) line bundle $\lineb$ is $\rho$\textit{-semiample} (or semiample over $Y$) if there exists a morphism $X \to X'$ over $Y$ and $m > 0$ such that $\lineb^{\otimes m}$ is the pull-back of a line bundle $M$ on $X'$ which is $\rho'$-ample, where $\rho' \colon X' \to Y$ is the structure morphism. 
        \item The line bundle $\lineb$ is $\rho$\textit{-movable} if the support of the cokernel of $\rho^* \rho_* \lineb^{\otimes m} \to \lineb^{\otimes m}$ has codimension at least two in $X$ for all sufficiently divisible positive integers $m$. 
        \item The \textit{N\'eron–Severi space} $N^1(X/Y )_{\Q} := \Pic(X/Y )_{\Q}/ \equiv$ is the vector space of numerical equivalence classes of line bundles, where $\lineb \equiv \lineb'$ if and only if $\deg(\lineb|_{\ell}) = \deg(\lineb'|_{\ell})$ for every proper curve $\ell$ in $X$ contracted by $\rho$. As noted in \cite[Page~291]{KawamataMatsudaMatsuki}, the $\Q$-vector space $N^1(X/Y )_{\Q}$ is finite-dimensional. 
        \item The \textit{relative ample cone} $\on{Amp}(X/Y )$ is the open cone in $N^1(X/Y )$ which consists of $\rho$-ample divisors.  
        \item The \textit{nef cone} of $X$ over $Y$ is the closed convex cone $\mr{Nef}(X/Y )$ in $N^1(X/Y)$ generated by line bundles on $X$ that are $\rho$-nef. The relative version of Kleiman’s ampleness criterion \cite[Theorem~0-1-2]{KawamataMatsudaMatsuki} implies that the relative ample cone $\mr{Amp}(X/Y )$ is the interior of $\mr{Nef}(X/Y )$.
        \item The \textit{movable cone} $\on{Mov}(X/Y)$ of $X$ over $Y$ is the closed convex cone in $N^1(X/Y)$ generated by (classes of) $\rho$-movable line bundles.
        \item A rational map $\phi \colon X \dashrightarrow X'$ (over $Y$) between normal varieties $X,X'$ projective over $Y$ is a \emph{rational contraction} if for some resolution (equivalently, for any resolution) $(p, q) \colon W \to X \times X'$ of the graph of $\phi$ such that $W$ is nonsingular, every $p$-exceptional effective divisor $E$ on $W$ satisfies $q_* \mathcal{O}_W(E) = \mathcal{O}_{X'}$ (see \cite[0.1~Definitions]{HuKeel} and \cite[Definition~2.8]{Ohta}). In the case the contraction is birational, this is equivalent to $E$ being $q$-exceptional.  
    \end{enumerate}
\end{defn}

If $\psi \colon X \dashrightarrow X'$ is a rational contraction and $\psi$ is birational then it is called a \textit{birational contraction}. 

\begin{defn}\label{defn:SQM}
   A \textit{small $\Q$-factorial modification} of $X$ is an algebraic fiber space $X' \to Y$, with birational contraction $\begin{tikzcd} X \ar[r,dotted,"\psi"] & X' \end{tikzcd}$ over $Y$ such that
\begin{enumerate}
    \item $X,X'$ are $\Q$-factorial,
    \item $\psi$ is an isomorphism in codimension one.
\end{enumerate}  
\end{defn}

Part (2) of \cref{defn:SQM} implies that for a resolution $(p, q) \colon W \to X \times X'$, an effective divisor $E$ is $p$-exceptional if and only if it is $q$-exceptional.

If $\phi \colon X \dashrightarrow X'$ is a birational map between normal varieties which is an isomorphism $\phi_U \colon U \to U'$, with $\mathrm{codim}_{X'} X' \setminus U' \ge 2$ then the \textit{strict transform along $\phi$} of a Weil divisor $D \subset X'$ is 
\[
\phi^* D := \overline{\phi_U^* (D _{U'})}.
\]
If $X \stackrel{\psi}{\dashrightarrow} X'$ is a small $\Q$-factorial modification then there is a canonical identification $\psi^* \colon \mathrm{Pic}(X'/Y)_{\Q} \iso \mathrm{Pic}(X/Y)_{\Q}$ given by taking the strict transform of the corresponding Weil divisor and realising some multiple of this as a Cartier divisor on $X'$. 

Let $Y$ be a normal variety. If $\rho \colon X \to Y$ is a proper birational morphism from a normal variety $X$ and $E \subset X$ a prime divisor such that $\rho(E)$ has codimension at least two in $Y$ then we say that $E$ is \textit{contracted} by $\rho$. The prime divisor $E$ defines a discrete valuation $\nu_E$ on the field $K(X) = K(Y)$. Conversely, given a discrete valuation $\nu$ on $K(Y)$, we say that $\nu$ is \textit{exceptional over $Y$} if there exists a projective birational morphism $\rho \colon X \to Y$, with $X$ normal, and a prime divisor $E \subset X$ contracted by $\rho$ such that $\nu = \nu_E$. Note that for a given $\nu$ there will be infinitely many pairs $(\rho \colon X \to Y,E)$ with $\nu = \nu_E$ since we can pick one pair and blowup away from $E$ to get another. In this way, we can also assume that $X$ is smooth; that is, $\rho$ is a resolution of singularities. If the canonical divisor $K_Y$ is $\Q$-Cartier then the discrepancy $a(E,Y)$ depends only on $\nu$ and not the choice of pair $(\rho \colon X \to Y, E)$. Thus, we can talk about the discrepancy $a(\nu,Y)$ of a discrete valuation $\nu$ exceptional over $Y$.    

\begin{lemma}\label{lem:Qfactorialterminaliffsmallmodification}
   Let \(Y\) be a normal variety whose canonical divisor $K_Y$ is $\Q$-Cartier and let $\rho \colon X \to Y$ be a projective \(\mathbb{Q}\)-factorial terminalization. Let $\xi \colon X'\to Y$ be a projective birational morphism, where \(X'\) is normal and \(\mathbb{Q}\)-factorial. Then \(\xi \) is a \(\mathbb{Q}\)-factorial terminalization of \(Y\) if and only if the induced birational map $\phi\colon X\dashrightarrow X'$ is a small $\Q$-factorial modification. 
\end{lemma}

\begin{proof}
Suppose first that \(\phi\) is an isomorphism in codimension one. Since
\(\rho\) is crepant, \(K_X=\rho^*K_Y\). Since $\phi$ is an isomorphism in codimension one, $K_{X}$ is the strict transform of $K_{X'}$ and hence $K_{X'}=\xi^* K_Y$ too. Thus, \( \xi\) is crepant.

Let \( \nu \) be a discrete valuation exceptional over \(X'\), meaning that there exists a proper birational morphism $f \colon Z \to X'$ with $Z$ normal, a prime divisor $E \subset Z$ with $f(E)$ of codimension at least two and $\nu = \nu_E$. Resolving $Z$ further if necessary, we may assume there also exists $g \colon Z \to X$ proper and birational. Since \(\phi\) is an
isomorphism in codimension one, \( \nu \) is also exceptional over \(X\): indeed, if the center $g(E)$ of $E$ on \(X\) were a prime divisor, its strict transform
on \(X'\) would be a prime divisor representing the same divisorial valuation $\nu$, contradicting the fact that \( \nu \) is exceptional over \(X'\). By crepancy,
\begin{equation*}
a(\nu,X')=a(\nu,Y)=a(\nu,X).
\end{equation*}
Since \(X\) is terminal and \( \nu \) is exceptional over \(X\), \(a(\nu,X)>0\) and hence \(a(\nu,X')>0\). This proves that \(X'\) is terminal. Therefore, \(\xi\) is a \(\mathbb{Q}\)-factorial terminalization.

Now suppose that \(\xi\) is a \(\mathbb{Q}\)-factorial terminalization. If $E \subset X$ is an exceptional divisor over \(Y\) then $a(E,Y)=0$ since $\rho$ is crepant. Conversely, if $\nu$ is a discrete valuation exceptional over \(Y\) with $a(\nu,Y)=0$ then $\nu$ cannot be exceptional over $X$: otherwise, the arguments above show that $a(\nu,X)=a(\nu,Y)=0$, which would contradict the fact that \(X\) is terminal. Thus, for $\nu$ a discrete valuation exceptional over \(Y\), $\nu = \nu_E$ for some $E \subset X$ if and only if $a(\nu,Y) = 0$. The same argument applies to \(X'\). Thus, the exceptional divisors of \(X \to Y\) and \(X' \to Y\) are the same. The non-exceptional prime divisors on $X$ and $X'$ are the strict transforms of prime divisors on \(Y\). Hence \(\phi\) contracts no divisor in either direction, and therefore is an isomorphism in codimension one.
\end{proof}

\begin{remark}\label{rem:Morifan}
Let \(Y\) be a normal \(\mathbb{Q}\)-Gorenstein variety and let $\rho \colon X\to Y$ be a projective \(\mathbb{Q}\)-factorial terminalization. 
Then one can show that $X \to Y$ is a relative Mori dream space. The rational contractions $\phi \colon X \dashrightarrow X'$ over $Y$ with $X'$ normal and projective over $Y$ are in one-to-one correspondence with the crepant projective partial resolutions of $Y$. Indeed, for a rational contraction $\phi \colon X \dashrightarrow X'$  \cite[Proposition~4.1]{Ohta} explains that there exists a small birational modification $X_i \to Y$ factoring as $X_i \to X' \to Y$. 
Since \cref{lem:Qfactorialterminaliffsmallmodification} says that $X_i \to Y$ is a  \(\mathbb{Q}\)-factorial terminalization, this implies that $X' \to Y$ is a projective crepant partial resolution. Conversely, we can extend a projective crepant partial resolution $X' \to Y$ to a $\Q$-factorial terminalisation $X_1 \to Y$ \cite{BCHM}. Then  \cref{lem:Qfactorialterminaliffsmallmodification} says that $X_1 \to Y$ is a small $\Q$-factorial modification of $X \to Y$, thus $X \dashrightarrow X'$ is a rational contraction.   
\end{remark}

Let $\tau \colon X \to X_0$ be a birational algebraic fiber space over $Y$. This means that all the morphisms $X \to Y, X_0 \to Y$ and $X \to X_0$ are algebraic fiber spaces. In particular, $X_0$ is also normal and $\tau$ is projective birational with $\tau_* \mc{O}_X = \mc{O}_{X_0}$. If the exceptional locus of $\tau$, denoted $\mathrm{Exc}(\tau)$, does not contain a divisor we call $\tau$ a \textit{small contraction}. 

Suppose $\tau$ is a small contraction and let $L$ be a line bundle on $X$ such that $L^{\otimes -1}$ is $\tau$-ample. The \textit{flip} of $\tau$ with respect to $L$ is a
commutative diagram
\[
\begin{tikzcd}
    X \ar[rr,dotted,"\psi"] \ar[dr,"\tau"'] & & X' \ar[dl,"{\tau'}"] \\
    & X_0 & 
\end{tikzcd}
\]
where $\tau'$ is also a small contraction, $\psi$ is an isomorphism in codimension one and the strict transform of $L$ along $\psi$ is $\tau'$-ample.

The existence and termination of flips is a subtle problem in general, see for instance \cite{BCHM}. But when they exist they are unique \cite[Proposition~5-1-11(2)]{KawamataMatsudaMatsuki}. 
We will always deal with (relative) Mori dream spaces, where the existence and termination of flips is known. Indeed, Mori dream spaces, as introduced by Hu-Keel \cite{HuKeel}, are projective varieties with properties specifically chosen so that the minimal model program can be run on them. In this paper, we are interested in relative Mori dream spaces as defined in \cite{Ohta}, where the base $Y$ will be assumed affine.

\begin{defn}\label{defn:MDS}
    Let $Y$ be an affine variety. Then $\rho \colon X \to Y$ is a 
\emph{relative Mori dream space} if $\rho$ is an algebraic fiber space 
and the following conditions hold:
\begin{enumerate}
  \item $X$ is $\mathbb{Q}$-factorial,
  \item $\operatorname{Pic}(X/Y)_{\mathbb{Q}} = N^{1}(X/Y)_{\mathbb{Q}}$, 
  \item $\mathrm{Nef}(X/Y)$ is a rational polyhedral cone generated by finitely many $\rho$-semiample divisors.
  \item there exist finitely many small $\mathbb{Q}$-factorial modifications $X_{1}, \dots, X_{k}$ over $Y$ satisfying (1),(2) and (3) such that the birational maps $g_{i} \colon X \dashrightarrow X_{i}$ satisfy  
  \begin{equation}\label{eq:movunion}
       \operatorname{Mov}(X/Y) = \bigcup_{i=1}^{k} g_{i}^{*} \operatorname{Nef}(X_{i}/Y).
  \end{equation}
\end{enumerate}
\end{defn}

In particular, $X,Y$ are normal and $\rho$ is projective since $\rho$ is assumed to be an algebraic fiber space; see \cref{defn:algfibrespace}. 

If $\phi \colon X \dashrightarrow X'$ is a rational contraction over $Y$ as in \cref{defn:picard_united_general_notation} (11) then \cite[Proposition~4.1]{Ohta} says that there exists a small $\Q$-factorial modification $X_i$ and a morphism $h \colon X_i \to X'$ such that $\phi = h \circ g_i$. Then $\phi^* = g_i^* \circ h^* \colon \Pic(X'/Y)_{\Q} \to \Pic(X/Y)_{\Q}$ is well-defined, and in fact is injective by \cite[Corollary~4.2]{Ohta}. Therefore, given a rational contraction $\phi$, the image $\phi^* \operatorname{Nef}(X'/Y)$ in $\Pic(X/Y)_{\Q}$ of the cone $\operatorname{Nef}(X'/Y)$ is a rational polyhedral cone. 

Let $\mathcal{M}(X/Y)$ be the set of all faces of all the rational polyhedral cones $g_{i}^{*} \operatorname{Nef}(X_{i}/Y)$ appearing in \eqref{eq:movunion}. Then it is proven in \cite[Theorem~4.4]{Ohta} that $\mathcal{M}(X/Y)$ is a fan whose support equals $\operatorname{Mov}(X/Y)$. Moreover, it is shown that $\phi \mapsto \phi^*\operatorname{Nef}(X'/Y)$ defines a bijection between the set of all rational contractions $\phi \colon X \dashrightarrow X'$ over $Y$ and $\mathcal{M}(X/Y)$. 
We call $\mathcal{M}(X/Y)$ the \textit{Mori fan} of the relative Mori dream space $X \to Y$. 

\begin{defn}
Let $\Lb = \{L_{1}, \dots, L_{n}\}$ be a collection of line bundles on $X$. The \emph{section ring} of $\Lb$ is
\[
  R(X, \Lb) = \bigoplus_{i \in \mathbb{Z}^{n}} 
  \Gamma\!\left(X, L^{\otimes i}\right),
\]
where $L^{\otimes i} = L_{1}^{\otimes i_{1}} \otimes \cdots \otimes L_{n}^{\otimes i_{n}}$. If $L_{1}, \dots, L_{n}$ form a basis of $\operatorname{Pic}(X/Y)_{\mathbb{Q}}$, then the section ring is called the \emph{Cox ring} of $X$ over $Y$. 
\end{defn}

A key result of \cite{Ohta}, upgrading the non-relative statement of \cite{HuKeel},  says that a  algebraic fiber space $\rho \colon X \to Y$ (with $Y$ being affine, $X$ being $\Q$-factorial and $\operatorname{Pic}(X/Y)_{\mathbb{Q}} = N^{1}(X/Y)_{\mathbb{Q}}$) is a relative Mori dream space if and only if the Cox ring of $X$ over $Y$ is finitely generated. For the remainder of this section, \textit{we assume that $\rho$ is a relative Mori dream space.}

Let $\Lb = \{L_{1}, \dots, L_{n}\}$ form a basis of $\operatorname{Pic}(X/Y)_{\mathbb{Q}}$ and denote $V = \operatorname{Spec}(R(X,\Lb))$.
Note that the $\Z^n$-grading on $R(X,\Lb)$, that comes from the basis $\Lb$ of $\operatorname{Pic}(X/Y)_{\mathbb{Q}}$, translates into an action of $T := \Gm^n$ on $V$. 
In this way the space of rational characters $\mathbb{X}^{\bullet}(T)_{\Q}$ is identified with $\operatorname{Pic}(X/Y)_{\mathbb{Q}}$.

The remainder of this section is devoted to the reconstruction of this identification that results in \cref{thm:LCNefMori}. While it may seem a tautology, rephrasing this map in the new way helps us identify the Mori fan of $X \to Y$ with a subfan of the GIT fan.

Crucially, if $\theta \in \mathbb{X}^{\bullet}(T)_{\Q}$ is mapped  to the class of a $\rho$-ample line bundle in $\operatorname{Pic}(X/Y)_{\mathbb{Q}}$ then \cite[Proposition~6.12]{Ohta} says that we can recover $X$ from $V$ as the GIT quotient $X_\theta$, 
\begin{equation} \label{eqn:variety_reconstruction_from_line_bundle}
    X_\theta := V \git_{\theta} T.
\end{equation}
Moreover, \text{loc. cit.} says that $\theta$ is generic with respect to $T$ acting on $V$. 

For each $\zeta \in \mathbb{X}^{\bullet}(T)$ we denote by $\CO_{\zeta}$ the trivial line bundle on $V$, twisted by $\zeta$. If $\zeta \in \mathbb{X}^{\bullet}(T)_{\Q} $ then we choose $m \in \Z_{> 0}$ such that $m \zeta \in \mathbb{X}^{\bullet}(T)$ and define 

\begin{equation*}
\alpha: \mathbb{X}^{\bullet}(T)_{\Q} \rightarrow  (\Pic_T(V)/\Pic(Y))_{\Q}; \hspace{5pt}   \alpha(\zeta) = \frac{1}{m} [\CO_{m\zeta}],
\end{equation*}
where $[\CO_{m\zeta}]$ denotes the class of the $T$-equivariant line bundle $\CO_{m\zeta}$ in $(\Pic_T(V)/\Pic(Y))_{\Q}$. 

Pull-back $q_{\theta}^*$ under the GIT quotient morphism $q_{\theta}$ maps line bundles $L \in \Pic(X_\theta)$ to $T$-equivariant line bundles $q_{\theta}^* L \in \Pic_T(V^{\theta})$. This map is injective by \cite[Proposition~5.1]{KKV} and induces an injection of vector spaces 
\[
q_{\theta}^* \colon \Pic(X_\theta / Y)_{\Q} \hookrightarrow (\Pic_T(V^{\theta}) / \Pic(Y))_{\Q}.
\]
Since $\theta$ is generic, the stabilizer of any point in $V^{\theta}$ is finite and only finitely many subgroups of $T$ can occur as stabilizers of points in $V^{\theta}$. Therefore, for any $\zeta \in \mathbb{X}^{\bullet}(T)_{\Q} $ there exists a sufficiently divisible integer $m$ such that $m \zeta \in \mathbb{X}^{\bullet}(T)$ and $m \zeta |_{T_p} = 1$ for all $p \in V^{\theta}$. Then Kempf's descent Lemma \cite[Proposition~4.2]{KKV} says that there exists a line bundle $L \in \Pic(X_\theta)$ such that $q_{\theta}^* L \cong \CO_{m\zeta}$ as $T$-equivariant line bundles. In other words, the morphism $q_{\theta}^*$ is also surjective. Hence, we obtain the \textit{linearisation} map
\begin{equation}\label{eq:Lthetamap}
    L_{\theta}: \mathbb{X}^{\bullet}(T)_{\Q} \rightarrow \Pic(X_\theta/Y)_{\Q}, \quad L_{\theta} = (q_{\theta}^*)^{-1} \circ \alpha. 
\end{equation}
We note that both $q_{\theta}$ and $\alpha$ only depend on the open set $V^{\theta}$ defined by $\theta$. Therefore, if $\theta' \in \mathbb{X}^{\bullet}(T)_{\Q} $ such that $V^{\theta} = V^{\theta'}$ then $X_{\theta} = X_{\theta'}$ and $L_{\theta} = L_{\theta'}$. In other words, $L_{\theta}$ depends only on the chamber $C$ in which $\theta$ lies. Therefore, we follow conventions and write $L_C := L_{\theta}$.



Let $\Sigma(V)$ be the GIT fan for the action of $T$ on $V$. As in \cite[Theorem~6.7]{Ohta}, which is a relative version of \cite[Theorem~2.3]{HuKeel}, we have: 

\begin{theorem}\label{thm:LCNefMori}
The linearisation map $L_C \colon \mathbb{X}^{\bullet}(T)_{\mathbb{Q}} \to \operatorname{Pic}(X/Y)_{\mathbb{Q}}$ is an isomorphism which identifies the closure $\overline{C}$ of the chamber $C$ with $\mathrm{Nef}(X/Y)$. Moreover, the Mori fan equals the subfan of $L_C(\Sigma(V)) := \{ L_C(\sigma) \, | \, \sigma \in \Sigma(V) \}$ consisting of all cones contained in the movable cone $\on{Mov}(X/Y)$. 
\end{theorem}

\begin{proof}
Since $V$ is of finite type, \cite[Corollary~6.14]{Ohta} says that the conditions of \cite[Theorem~6.7]{Ohta} are satisfied. Therefore, $L_C$ is an isomorphism. Moreover, the fact that $L_C$ identifies $\overline{C}$ with $\mathrm{Nef}(X/Y)$ and the movable fan with the subfan of $L_C(\Sigma(V))$ supported on $\on{Mov}(X/Y)$ are both shown in the proof of \cite[Theorem~6.7]{Ohta}.
\end{proof}

\subsection{Symplectic singularities}

\begin{defn}
A variety $Y$ is called a symplectic singularity if it is normal and its smooth locus $U$ admits an everywhere-nondegenerate closed $2$-form $\omega$ such that for an arbitrary resolution $f \colon Z \rightarrow Y$ of singularities $\omega$ extends to a regular $2$-form on $Z$. 
\end{defn}

Denote by $\Sigma$ the singular locus of $Y$. Rather than use directly the definition of terminal singularities given in \cref{defn:do-not-use-terminal}, we will use the following equivalent characterization for varieties with symplectic singularities, which is due to Namikawa \cite[Corollary~1]{NamikawaNote}.

\begin{thm} \label{thm:defn_terminal}
    A symplectic singularity $Y$ has terminal singularities if and only if the codimension of $\Sigma$ in $Y$ is at least $4$.
\end{thm} 

\begin{defn}\label{dfn:partial_symplectic_resolution}
    A morphism $f\colon X \rightarrow Y$ of normal varieties is called a partial resolution if $f$ is proper and birational and the restriction of $f$ to $f^{-1}(Y_{\on{reg}})$ is an isomorphism. It is also symplectic if the symplectic form on $f^{-1}(Y_{\on{reg}})$ extends to a symplectic form on $X_{\on{reg}}$.
\end{defn}
If $f \colon Z \to X$ is a symplectic partial resolution then \cite[Lemma 4.3]{Alberto_Malaney_partial_resolutions} says that $Z$ also has symplectic singularities. Moreover, it is shown in \cite[Theorem 4.4]{Alberto_Malaney_partial_resolutions} that a partial resolution of a symplectic singularity is symplectic if and only if it is crepant.

The following technical lemma will be used in the proof of the \cref{prop:Poissomorphismsympsing}.
\begin{lemma}\label{lem:techincainj}
    Let $X$ be an irreducible variety and $\mc{F},\mc{G}$ locally free sheaves on $X$. Suppose $A$ and $B$ are two $\mc{O}$-module morphisms $A \colon \mc{F} \to \mc{G}$ and $B \colon \mc{G} \to \mc{F}$ that are inverse isomorphisms at the generic point of $X$. Then $A,B$ are isomorphisms. 
\end{lemma}

\begin{proof}
    Note that $A,B$ are isomorphisms if they are isomorphisms $\Gamma(Z_i,\mc{F}) \stackrel{\rightarrow}{\leftarrow} \Gamma(Z_i,\mc{G})$ for some affine open cover $\{ Z_i \}$ of $X$. Therefore, we may assume that $X$ is affine and $\mc{F},\mc{G}$ are free. Consider $C := \Id - B \circ A \colon \mc{F} \to \mc{F}$. This is a morphism from the free $\mc{O}_X$-module $\mc{F}$ to itself that is zero at the generic point $\eta$. But we have a commutative diagram 
    \[
    \begin{tikzcd}
    \mc{F}_{\eta} \ar[r,"C = 0"] & \mc{F}_{\eta} \\
    F \ar[u,hook] \ar[r,"C"] & F \ar[u,hook] ,
    \end{tikzcd}
    \]
    where $F = \Gamma(X,\mc{F})$. This forces $C = 0$. 
\end{proof}

\begin{proposition}\label{prop:Poissomorphismsympsing}
    Let $Y$ be a variety with symplectic singularities, $X$ a normal Poisson variety and $\pi \colon X \to Y$ a proper birational Poisson morphism. Then $X$ has symplectic singularities and $\pi$ is a crepant partial resolution. 
\end{proposition}

\begin{proof}
Let $\sigma_Y$ be the symplectic form on the regular locus of $Y$. First we argue that $X$ has symplectic singularities. We begin by noting that $\sigma_X := \pi^* \sigma_Y$ is a regular (closed) $2$-form on the smooth locus $U$ of $X$. Indeed, we can fix a resolution of singularities $f \colon Z \to X$ such that $f$ is an isomorphism over $U$. Then $f^* \sigma_X = (\pi \circ f)^* \sigma_Y$ is a regular form on $Z$ because $Y$ is assumed to have symplectic singularities. But $f$ is an isomorphism over $U$ which implies that $\sigma_X = \pi^* \sigma_Y$ must be regular on $U$. Next, we claim that $\sigma_X$ is non-degenerate on $U$. It defines a morphism of $\mc{O}_{U}$-modules $\phi \colon \Theta_{U} \to \Omega^1_{U}$ given by $\phi(v)(w) = \sigma_X(v,w)$. The Poisson bracket on $\mc{O}_X$ similarly defines a morphism $\psi \colon \Omega^1_{U} \to \Theta_{U}$. Note that the Poisson bracket on $Y$ is induced from the symplectic form $\sigma_Y$. Since $\pi$ is Poisson and generically an isomorphism, this means that $\phi,\psi$ are inverse isomorphisms at the generic point of $U$. Since $\Omega^1_{U}, \Theta_{U}$ are locally free and $U$ smooth, this implies that they are everywhere isomorphisms by \cref{lem:techincainj}. Hence $\sigma_X$ is non-degenerate. As noted previously, $f^* \sigma_X = (\pi \circ f)^* \sigma_Y$ is a regular form on $Z$. Thus, $X$ has symplectic singularities. 

A partial resolution is a proper birational morphism. Therefore, we only need to show that it is crepant. That is, if ${\omega}_Y$ is the canonical sheaf (a line bundle in this case) on $Y$ then we wish to show that $\pi^* {\omega}_Y \cong {\omega}_X$. But, in fact, since $Y$ and $X$ have symplectic singularities, we have ${\omega}_Y \cong \mc{O}_Y$ and ${\omega}_X \cong \mc{O}_X$. Therefore, $\pi^* {\omega}_Y = \pi^* \mc{O}_Y = \mc{O}_X = {\omega}_X$. 

By our definition of crepant partial resolution, we also need to show that $\pi$ is an isomorphism over $V:=Y_{\mathrm{reg}}$. The restricted morphism $\pi_V\colon \pi^{-1}(V)\rightarrow V$ is projective and birational. Since $V$ is smooth, it is
terminal; hence crepancy implies that $\pi_V$ has no exceptional prime divisors. Moreover, $V$ is locally factorial, and a projective birational morphism to a normal
$\mathbb Q$-factorial variety with no exceptional prime
divisors is an isomorphism by Van der Waerden's purity theorem. Therefore, $\pi_V$ is an
isomorphism, and $\pi$ is a crepant partial resolution.

\end{proof}



\begin{thm}\label{thm:XYMDS}
Assume that an affine variety $Y$ has symplectic singularities. If $X \to Y$ is a $\Q$-factorial terminalization then it is a relative Mori dream space.
\end{thm}

\begin{proof}
As noted in \cite[Lemma~1]{Namikawa3}, this follows from \cite[Corollary~1.3.2]{BCHM}. See also \cite[Example~7.1]{Ohta}. 
\end{proof}

Not only do $\Q$-factorial terminalizations $X \to Y$ of symplectic singularities give rise to relative Mori dream spaces, these are also the maximal crepant projective partial resolutions of $Y$; specifically, if $Z \to X$ is a crepant projective partial resolution with $Z$ normal then $Z \to X$ is an isomorphism; see for instance \cite[Lemma 4.5]{Alberto_Malaney_partial_resolutions}.

  \section{Quotients of crepant resolutions} \label{sec:5.Quotients}

In this section, we consider when the quotient of a crepant partial resolution $X \to Y$ by a finite group $\Gamma$ gives rise to a $\Q$-factorial terminalization $X / \Gamma \to Y / \Gamma$. 

\subsection{Group quotients}

Let $Y$ be a normal affine variety. Assume that the algebraic fiber space $\rho \colon X \to Y$ is a relative Mori dream space. Let $\Gamma$ be a finite group acting on $X$ and $Y$ such that $\rho \colon X \to Y$ is $\Gamma$-equivariant. Let $\pi \colon X \to X/\Gamma$ be the quotient morphism.  

In addition to (1)--(4) of \cref{defn:MDS}, we assume that :
\begin{enumerate}
\item[(5)] the pullback by $\pi$ defines an isomorphism $\pi^{*} \colon \operatorname{Pic}((X/\Gamma) / (Y/\Gamma))_{\mathbb{Q}} \;\xrightarrow{\;\sim\;}\; \operatorname{Pic}(X/Y)_{\mathbb{Q}}$.
\end{enumerate}
This condition is a rephrasing of the \cref{lem:pullbackgroupiso}(iii) applied to the relative setting. In particular it follows for the case of $\Gamma$ being a subgroup of a connected algebraic group action from applying \cref{lem:pullbackgroupiso}(iii) twice; to $X$ and $Y$ respectively. 

Since $N^1(X/Y)$ is finite-dimensional and $\operatorname{Pic}(X/Y)_{\mathbb{Q}} \cong N^1(X/Y)$, assumption (5) implies that $\operatorname{Pic}((X/\Gamma) / (Y/\Gamma))_{\mathbb{Q}}$ is finite-dimensional. We fix line bundles $L_{1}, \dots, L_{n}$ on $X/\Gamma$ such that their classes form a basis of the vector space $\operatorname{Pic}((X/\Gamma) / (Y/\Gamma))_{\mathbb{Q}}$.

\begin{lemma}\label{lem:intlinebundle}
    Let $L$ be a line bundle on $X/\Gamma$. Then $\Gamma(X/\Gamma, L) = \Gamma(X,\pi^* L)^\Gamma$.
\end{lemma}

\begin{proof}
If $U \subset X/\Gamma$ is an affine open subset then note that $\pi^{-1}(U)$ is also affine because $\pi$ is an affine morphism. Then $\Gamma(\pi^{-1}(U),\pi^* L) = \mathcal{O}(\pi^{-1}(U)) \otimes_{\mathcal{O}(U)} L(U)$ and hence the space of invariant global sections $\Gamma(\pi^{-1}(U),\pi^* L)^\Gamma$ equals $L(U)$. Then $\Gamma(X,\pi^* L)$ can be computed via the \v{C}ech complex associated to an affine open cover of $X/\Gamma$. Taking invariants gives the \v{C}ech complex for $\Gamma(X/\Gamma,L)$. 
\end{proof}

\begin{theorem}\label{thm:XWMDS}
$X/\Gamma$ is a relative Mori dream space over $Y/\Gamma$.
\end{theorem}

\begin{proof}
First, we note that $X/\Gamma$ is $\mathbb{Q}$-factorial by \cref{lem:Qfactorialquotient} because $X$ is $\mathbb{Q}$-factorial.  

By \cref{lem:intlinebundle}, the Cox ring
\begin{equation}\label{eq:Coxinvariant}
      R(X/\Gamma, \mathcal{L}) 
  = R\!\left(X, \pi^{*}\mathcal{L}\right)^{\Gamma}
\end{equation}
is finitely generated because $R(X, \pi^{*}\mathcal{L})$ is finitely generated by \cite[Proposition~6.10]{Ohta}.  

Finally, we must show that $\operatorname{Pic}((X/\Gamma) / (Y/\Gamma))_{\mathbb{Q}} = N^{1}((X/\Gamma) / (Y/\Gamma))_{\mathbb{Q}}$. In other words, if $L$ is a numerically trivial line bundle on $X/\Gamma$ then $[L] = 0$ in $\operatorname{Pic}((X/\Gamma) / (Y/\Gamma))_{\mathbb{Q}}$. Since we assume (5) that $\pi^* \colon \operatorname{Pic}((X/\Gamma) / (Y/\Gamma))_{\mathbb{Q}} \;\xrightarrow{\;\sim\;}\; \operatorname{Pic}(X/Y)_{\mathbb{Q}}$ and we know from the equality $\operatorname{Pic}(X / Y)_{\mathbb{Q}} = N^{1}(X / Y)_{\mathbb{Q}}$ that if $M$ is a numerically trivial line bundle on $X$ that $[M]=0$ in $\operatorname{Pic}(X / Y)_{\mathbb{Q}}$, it suffices to show that if $L$ is a numerically trivial line bundle on $X/\Gamma$ then $\pi^* L$ is numerically trivial on $X$. Let $\ell \subset X$ be a (proper) curve contracted by $\rho$. Let $\widetilde{\ell}$ be the image of $\ell$ in $X/\Gamma$. Since $\pi$ is a finite morphism, $\widetilde{\ell}$ is a curve. We have a commutative diagram 
\begin{equation}\label{eq:curvemap}
    \begin{tikzcd}
    \ell \ar[r,hook] \ar[d,"\pi_{\ell}"'] & X \ar[d,"\pi"] \\
    \widetilde{\ell} \ar[r,hook] & X/\Gamma
\end{tikzcd}
\end{equation}
where $\pi_{\ell}$ is finite because $\pi$ is a finite morphism. Then the line bundle $(\pi^* L) |_{\ell} = \pi_{\ell}^* (L |_{\widetilde{\ell}})$ has degree $e \deg (L |_{\widetilde{\ell}})$ by \cite[Lemma~0AYZ]{stacks}, where $e$ is the degree of the morphism $\pi_{\ell}$. Therefore, if $\deg (L |_{\widetilde{\ell}}) = 0$ then $\deg (\pi^* L) |_{\ell} = 0$ too and thus $\pi^* L$ is numerically trivial. 

It follows that $X/\Gamma \to Y/\Gamma$ is a relative Mori dream space by \cite[Theorem~1.1]{Ohta}.
\end{proof}

Let $T$ be the torus associated to the relative Mori dream space $X \to Y$ as in \cref{sec:birationaldefn}. Then $T$ acts on the Cox ring of $X/Y$. \cref{thm:LCNefMori} says that the linearization map $L_C \colon \mathbb{X}^{\bullet}(T)_{\Q} \to \operatorname{Pic}(X/Y)_{\mathbb{Q}}$ is an isomorphism. 

\begin{lemma}\label{lem:LCLCWcomm}
There is a commutative diagram
    \[
\begin{tikzcd}
    & \mathbb{X}^{\bullet}(T)_{\mathbb{Q}} \ar[dr,"L_C"] \ar[dl,"L_C^\Gamma"'] & \\
  \ar[rr,"{\pi^{*}}"',"\sim"] \operatorname{Pic}((X/\Gamma) / (Y/\Gamma))_{\mathbb{Q}} & & \operatorname{Pic}(X/Y)_{\mathbb{Q}} .
\end{tikzcd}
\] 
\end{lemma}

\begin{proof}
Let $\zeta \in \mathbb{X}^{\bullet}(T)_{\mathbb{Q}}$ and choose $m \in \Z_{> 0}$ such that $m \zeta \in  \mathbb{X}^{\bullet}(T)$ and $m\zeta |_{T_p} = 1$ for all $p \in V^{\theta}$, as was done in the construction of the map \eqref{eq:Lthetamap}. There is a commutative diagram 
\[
\begin{tikzcd}
V^{\theta} \ar[r,"q_{\theta}"] \ar[d,"\tilde{\pi}"'] & V^{\theta} \git T = X \ar[d,"\pi"] \\
V^{\theta} / \Gamma \ar[r,"Q_{\theta}"] & X / \Gamma. 
\end{tikzcd}
\]
Then $L_C^{\Gamma}(m \zeta) = (Q_{\theta}^*)^{-1}(\mathcal{O}_{V^{\theta} / \Gamma} \otimes m \zeta)$ and $L_C(m \zeta) = (q_{\theta}^*)^{-1}(\mathcal{O}_{V^{\theta}} \otimes m \zeta)$; equivalently, 
\[
Q_{\theta}^* \, L_C^{\Gamma}(m \zeta) = \mathcal{O}_{V^{\theta} / \Gamma} \otimes m \zeta \quad \textrm{ and } \quad q_{\theta}^* \, L_C(m \zeta) = \mathcal{O}_{V^{\theta}} \otimes m \zeta.
\]
Note that $\mathcal{O}_{V^{\theta}}  = \tilde{\pi}^* \mathcal{O}_{V^{\theta} / \Gamma} $ as $T$-equivariant bundles and hence $\tilde{\pi}^*(\mathcal{O}_{V^{\theta} / \Gamma} \otimes m \zeta) \cong \mathcal{O}_{V^{\theta}} \otimes m \zeta$. Therefore, 
\begin{align*}
q_{\theta}^* L_C(m \zeta) & = \tilde{\pi}^*(\mathcal{O}_{V^{\theta} / \Gamma} \otimes m \zeta) \\
& = \tilde{\pi}^* (Q_{\theta}^* L_C^{\Gamma}(m \zeta) ) \\
& = q_{\theta}^*(\pi^* L_C^{\Gamma}(m \zeta)). 
\end{align*}
Since $q_{\theta}^* \colon \Pic(X_\theta / Y)_{\Q} \stackrel{\sim}{\rightarrow} (\Pic_T(V^{\theta}) / \Pic(Y))_{\Q}$  is an isomorphism, this implies that $L_C(m \zeta) = \pi^* L_C^{\Gamma}(m \zeta)$, as required. 

\end{proof}

Let $X_1 \ds X_k$ be the small $\Q$-factorial modifications of $X$ over $Y$. 

\begin{theorem}\label{thm:pistarmorifan}
The isomorphism $\pi^{*} \colon \operatorname{Pic}((X/\Gamma) / (Y/\Gamma))_{\mathbb{Q}} \;\xrightarrow{\;\sim\;}\; \operatorname{Pic}(X/Y)_{\mathbb{Q}}$ satisfies: 
\begin{enumerate}[label=(\roman*), ref=(\roman*)]
    \item[(i)] \label{thm:pistarmorifan:i} $\pi^* \mathrm{Nef}((X/\Gamma)/(Y/\Gamma)) = \mathrm{Nef}(X/Y)$; 
    \item[(ii)] \label{thm:pistarmorifan:ii} $\pi^* \mathrm{Mov}((X/\Gamma)/(Y/\Gamma)) = \mathrm{Mov}(X/Y)$; 
    \item[(iii)] \label{thm:pistarmorifan:iii} $\pi^*$ identifies the Mori fan of $X/\Gamma$ over $Y/\Gamma$ with the Mori fan of $X$ over $Y$; and  
    \item[(iv)] \label{thm:pistarmorifan:iv} The small $\Q$-factorial modifications of $X/\Gamma$ over $Y/\Gamma$ are $X_i / \Gamma$ for $i = 1 \ds k$. 
\end{enumerate}
\end{theorem}

\begin{proof}
Part (i). This is very similar to the proof of \cref{thm:XWMDS}. Let $L \in \operatorname{Pic}((X/\Gamma) / (Y/\Gamma))$ be $(\rho/\Gamma)$-nef and let $\ell \subset X$ be a curve contracted by $\rho$. Then, for $\tilde{\ell}$ as in diagram \eqref{eq:curvemap}, $\deg (\pi^* L) |_{\ell} = e \deg (L |_{\tilde{\ell}}) \ge 0$ and hence $\pi^* L$ is $\rho$-nef. Conversely, let $\pi^* L$ be $\rho$-nef and $\tilde{\ell}$ a curve contracted by $\rho/\Gamma$. We choose an irreducible component $\ell \subset \pi^{-1}(\tilde{\ell})$. Since $\pi$ is finite, $\ell$ is a curve with $\pi(\ell) = \tilde{\ell}$. Hence, as before, $0 \le \deg (\pi^* L) |_{\ell} = e \deg (L |_{\tilde{\ell}})$ and hence $L$ is $(\rho/\Gamma)$-nef.

Part (ii). Since the movable cone is defined to be the closure of the set of movable line bundles and $\pi^*$ is an isomorphism on relative Picard groups, it suffices to show for $L \in \operatorname{Pic}((X/\Gamma) / (Y/\Gamma))$ that $L$ is movable if and only if $\pi^* L$ is movable. Let $\eta := \rho/\Gamma \colon X/\Gamma \to Y/\Gamma$ so that $\rho = \eta \circ \pi$. If $L$ is a line bundle on $X/\Gamma$, we let $B_{Y/\Gamma}(L)$ denote the set of points in $X/\Gamma$ where the counit of adjunction $\eta^* \eta_* L \to L$ is not surjective; this is the relative base locus of $L$. Similarly, if $M$ a line bundle on $X$ then $B_Y(M)$ is the relative base locus of $M$. We claim that for every line bundle $L$ on $X/\Gamma$, $B_Y(\pi^* L) \subset \pi^{-1} B_{Y/\Gamma}(L)$ and for all $m$ divisible by $|\Gamma|$, $B_Y(\pi^* L^{\otimes m}) = \pi^{-1} B_{Y/\Gamma}(L^{\otimes m})$. 

Both claims are local on $Y$. Therefore, we may assume that $Y = \Spec A$ is affine and $Y/\Gamma = \Spec A^{\Gamma}$. Then $\rho^* \rho_* (\pi^* L) \cong \mathcal{O}_X \otimes_A \Gamma(X,\pi^* L)$ such that the morphism $\rho^* \rho_* (\pi^* L) \to \pi^* L$ equals the multiplication map $\mathcal{O}_X \otimes_A \Gamma(X,\pi^* L) \to \pi^* L$. In particular, $B_Y(\pi^* L)$ is the (closed, since $\rho$ is projective) set of all points in $X$ where all global sections $s \in \Gamma(X,\pi^* L)$ vanish. Certainly, $A \otimes_{A^{\Gamma}} \Gamma(X/\Gamma,L) \subset \Gamma(X,\pi^* L)$. Therefore, if $x'\notin B_{Y/\Gamma}(L)$ then there exists $s \in \Gamma(X/\Gamma,L)$ with $s(x') \neq 0$ and hence $(1 \otimes s)(x) \neq 0$ for all $x \in \pi^{-1}(x')$. This implies that $B_Y(\pi^* L) \subset \pi^{-1} B_{Y/\Gamma}(L)$. 

For the inclusion $\pi^{-1} B_{Y/\Gamma}(L^{\otimes m}) \subset B_Y(\pi^* L^{\otimes m})$, let $x \in B_Y(\pi^* L)$; this means that we can choose $s \in \Gamma(X,\pi^* L)$ such that $s(x) \neq 0$. Then, for each $g \in \Gamma$, $(g \cdot s)(g \cdot x) = s(x) \neq 0$. Let 
\[
E = \mathrm{Span}_{\Q}\{ g \cdot s \, | \, g \in \Gamma \} \subset \Gamma(X,\pi^* L).
\]
Then for each $z \in G \cdot x$ the functional $\mathrm{ev}_z \colon E \to \C$, $\mathrm{ev}_z(t) = t(z)$ has kernel a hyperplane. Since $\Gamma \cdot x$ is a finite set we can therefore pick $t \in E \setminus \bigcup_{z \in \Gamma \cdot x} \Ker \mathrm{ev}_z$. Then $\tilde{t} = \prod_{g \in \Gamma} g \cdot t$ is a global section of $\pi^* L^{ \otimes |\Gamma|}$ that is $\Gamma$-invariant and doesn't vanish on any point in $\Gamma \cdot x$. Lemma~\ref{lem:intlinebundle} says that $\tilde{t} \in \Gamma(X,\pi^* L)^{\Gamma} = \Gamma(X/\Gamma,L)$. Since $\tilde{t}(\pi(x)) \neq 0$, this shows that $\pi^{-1} B_{Y/\Gamma}(L^{\otimes |\Gamma|}) \subset B_Y(\pi^* L^{\otimes |\Gamma|})$. 

Part (iii). If $V = \Spec R(X,\mathcal{L})$ and $\theta \in \mathbb{X}^{\bullet}(T)$ such that $X = V \git_{\theta} T$ then \eqref{eq:Coxinvariant} shows that $X/\Gamma = (V/\Gamma) \git_{\theta} T$. Moreover, \cref{prop:GITfansthesame} says that we have an equality of GIT cones $\Sigma(V) = \Sigma(V/\Gamma)$. Therefore, \cref{thm:LCNefMori} says in this case that:
\[
\mathcal{M}(X/Y) = \{ \sigma \in L_C(\Sigma(V)) \, | \, \sigma \subset \mathrm{Mov}(X/Y) \},
\]
\[
\mathcal{M}((X/\Gamma)/(Y/\Gamma)) = \{ \sigma \in L_C^{\Gamma}(\Sigma(V)) \, | \, \sigma \subset \mathrm{Mov}((X/\Gamma)/(Y/\Gamma)) \}.
\]
Therefore, part (iii) follows from \cref{lem:LCLCWcomm} and part (ii). 

Since $X/\Gamma \to Y/\Gamma$ is a relative Mori dream space, the small $\Q$-factorial modifications of $X/\Gamma$ over $Y/\Gamma$ are the spaces $Z_i = \Proj \bigoplus_{n \ge 0} \Gamma(X/\Gamma,M_i^{\otimes n})$, where $M_i$ is an element in the interior of the Mori chamber $\sigma_i$ of the Mori fan of $X/ \Gamma \to Y/\Gamma$. But part (iii) says that $\pi^* \sigma_i$ are precisely the Mori chambers for $X \to Y$. Thus, up to relabelling, $X_i = \Proj \bigoplus_{n \ge 0} \Gamma(X,(\pi^* M)_i^{\otimes n})$. Since $\Gamma$ acts on each space $\Gamma(X,(\pi^* M)_i^{\otimes n})$, it acts by algebra automorphisms on $\bigoplus_{n \ge 0} \Gamma(X,(\pi^* M)_i^{\otimes n})$ and hence on $X_i$ too. Moreover, by \cref{lem:grouprojquotient}, we have  $Z_i = X_i /\Gamma$, as required.  
\end{proof}

\subsection{Quotients of symplectic singularities}

Throughout this section, let $Y$ be a symplectic singularity. 

\begin{lemma}\label{lem:surfaceblowup}
    Let $S$ be a surface with a Kleinian singularity. 
    \begin{enumerate}
        \item[(i)] If $p \colon S' \to S$ is the blowup of the reduced singular locus then $p$ is a crepant projective partial resolution of $S$. 
        \item[(ii)] The repeated blowup of the reduced singular locus eventually terminates in the minimal resolution $\widetilde S\to S$. 
    \end{enumerate}
\end{lemma}

\begin{proof}

Part (i). Crepancy (in slightly different terms) and normality of the blowup are explained for example in \cite[Prop. 5.1]{ArtinLipman1986}. Blowing up is a projective and birational morphism, hence it is a crepant projective partial resolution.

Part (ii) is explained in \cite[\S 6]{ArtinLipman1986} (see also \cite[\S 54.12]{stacks}).


\end{proof}

For a variety $Y$ with symplectic singularities,  we define (compare Reid \cite[1.14]{Reid3fold}) $Y_1 \subset Y$ to be the set of points $y$ such that either $Y$ is smooth at $y$ or the analytic neighbourhood of $y$ in $Y$ is isomorphic to 
\begin{equation}\label{lem:functorial-resolution-codim-two.1}
        U_y    \simeq B\times S,
\end{equation}
where $B$ is a polydisk and $S$ a Kleinian singularity with singular point $0$. The complement $\mathrm{Diss}(Y)$ to $Y_1$ in $Y$ is called the \emph{dissident locus} in loc. cit.. The following folklore result is stated in several places, but we do not know of any proof in the literature. We included a proof since the result is key to the proof of \cref{prop:Gactionlifts}. 

\begin{lemma}\label{lem:functorial-resolution-codim-two}
Let $Y$ be a variety with symplectic singularities and $\rho\colon X\to Y$ be a crepant partial resolution such that $X$ has terminal singularities. Let $X_1:=\rho^{-1}(Y_1)$. Then the following hold.
\begin{enumerate}[label=(\roman*), ref=(\roman*)]
    \item \label{lem:functorial-resolution-codim-two:i} The complement $Y\setminus Y_1$ has codimension at least four.

    \item \label{lem:functorial-resolution-codim-two:ii}Starting from $Y_1^{(0)}:=Y_1$ and defining inductively
    \[
        Y_1^{(i+1)}
        :=
        \mathrm{Bl}_{Z_i}(Y_1^{(i)}),
        \qquad
        Z_i:=\mathrm{Sing}(Y_1^{(i)})_{\mathrm{red}},
    \]
    whenever $Y_1^{(i)}$ is singular, the process terminates after
    finitely many steps in a smooth variety $\widetilde Y_1$. The
    resulting morphism $\pi_1\colon\widetilde Y_1\to Y_1$ is a projective crepant, hence symplectic, resolution.

    \item \label{lem:functorial-resolution-codim-two:iii} The morphism $\pi_1$ is the unique projective crepant
    resolution of $Y_1$, up to a unique isomorphism over $Y_1$.

    \item \label{lem:functorial-resolution-codim-two:iv}The restriction $\rho_1:=\rho|_{X_1}\colon X_1\rightarrow Y_1$ is a projective crepant resolution. Consequently there is a unique
    isomorphism $\alpha\colon X_1\xrightarrow{\sim}\widetilde Y_1$ over $Y_1$.

    \item \label{lem:functorial-resolution-codim-two:v}Every algebraic group acting algebraically on $Y$ preserves
    $Y_1$, and its action lifts uniquely to an algebraic action on $X_1$, for which $\rho_1$ is equivariant.
\end{enumerate}
\end{lemma}

\begin{proof}

 \ref{lem:functorial-resolution-codim-two:i} \hspace{0.5pt} This follows immediately from the fact that $Y$ has a finite stratification by symplectic leaves, which are even dimensional and Kaledin's formal neighbourhood theorem \cite{Kaledinsympsingularities} which implies that $\mathrm{Diss}(Y)$ is the union of all leaves of codimension at least four.



\medskip

\noindent
\ref{lem:functorial-resolution-codim-two:ii} \hspace{0.5pt}
Set $Y_1^{(0)}:=Y_1$. If $Y_1^{(i)}$ is singular, let $Z_i:=\mathrm{Sing}(Y_1^{(i)})_{\mathrm{red}}$ and define
\begin{equation}\label{lem:functorial-resolution-codim-two.3}
        \beta_i\colon
    Y_1^{(i+1)}
    :=
    \mathrm{Bl}_{Z_i}(Y_1^{(i)})
    \to
    Y_1^{(i)}.
\end{equation}
In the analytic product decomposition \eqref{lem:functorial-resolution-codim-two.1}, the center of the first
blowup is $\mathrm{Sing}(U_y)_{\mathrm{red}}=B\times\{0\}$. Since $\mathrm{Bl}_{B\times\{0\}}(B\times S)
    \simeq    B\times\mathrm{Bl}_0(S)$, the reduced singular locus of $\mathrm{Bl}_{B\times\{0\}}(B\times S)$ is $B\times
    \mathrm{Sing}(\mathrm{Bl}_0(S))_{\mathrm{red}}$. Therefore, \cref{lem:surfaceblowup} implies statement \ref{lem:functorial-resolution-codim-two:ii}.
\medskip

\noindent
\ref{lem:functorial-resolution-codim-two:iii} \hspace{0.5pt} This claim follows from \cite[Corollaries 3.38, 3.39]{KaplanSchedler_crepant_gluing_stratified}. 

\medskip

\noindent
\ref{lem:functorial-resolution-codim-two:iv} \hspace{0.5pt}
Since $X$ is a terminal symplectic variety, $\mathrm{codim}_X \, \mathrm{Sing}(X)~\geq~4$. We claim that $\mathrm{Sing}(X)\cap X_1=\emptyset$. Suppose otherwise, and consider the closed subset
\[
    T:=
    \rho\bigl(\mathrm{Sing}(X)\cap X_1\bigr)
    \subseteq Y_1.
\]
It is closed because $\rho$ is proper.

The varieties $X$ and $Y$ are Poisson and $\rho$ is a Poisson morphism. Moreover, the singular locus of a symplectic variety is a Poisson subvariety \cite{Kaledinsympsingularities}. It follows that $T$ is a Poisson subvariety of $Y_1$. Consequently $T$ is a union of symplectic leaves of $Y_1$. As shown in the proof of part \ref{lem:functorial-resolution-codim-two:i}, every symplectic leaf of $Y_1$ has codimension either zero or two in $Y$. If $T$ is nonempty, it therefore contains a symplectic leaf $\mc{L}$ satisfying $\dim \mc{L} \geq\dim Y-2$. Since $\mc{L} \subseteq \rho\bigl(\mathrm{Sing}(X)\cap X_1\bigr)$, we obtain
\[
    \dim\mathrm{Sing}(X)     \geq    \dim \mc{L}   \geq     \dim X-2,
\]
contradicting $\mathrm{codim}_X \, \mathrm{Sing}(X)\geq4$. This proves that $\mathrm{Sing}(X)\cap X_1=\emptyset$. Therefore $X_1$ is smooth. Since $\rho$ is crepant, its restriction $\rho_1 \colon X_1 \to Y_1$ is a projective crepant resolution.

\medskip

\noindent
\ref{lem:functorial-resolution-codim-two:v}
To obtain a $G$-action on $X_1$ we will lift the $G$-action from $Y_1$ to $\widetilde{Y}_1$ inductively and use the isomorphism in step \ref{lem:functorial-resolution-codim-two:iv}. Every automorphism of $Y$ preserves the local analytic type of the singularity. Thus, for a group $G$ acting on $Y$, the open set $Y_1$ is $G$-stable. 

Suppose that the $G$-action lifts to $Y_1^{(i)}$. The reduced singular locus $Z_i$ that we were blowing up in step (ii) is $G$-stable. This means that the corresponding ideal sheaf $\mathcal{I}_i$ is $G$-equivariant and hence so too is the Rees algebra 
\[
\mathcal{R}(\mathcal{I}_i) = \bigoplus_{m \ge 0} \mathcal{I}_i^m.
\]
 Since $Y_1^{(i+1)}=\on{Proj}_{Y_1^{(i)}}\mathcal{R}(\mathcal{I}_i)$, the $G$-action lifts to an action on
$Y_1^{(i+1)}$ and the blowup $\beta_i\colon Y_1^{(i+1)} \to Y_1^{(i)}$ is $G$-equivariant.
Iterating gives a $G$-action on $\widetilde Y_1$ for which $\pi_1\colon\widetilde Y_1\rightarrow Y_1$ is $G$-equivariant. By part \ref{lem:functorial-resolution-codim-two:iv}, we get a $G$-action on $X_1$ with $\rho_1$ being $G$-equivariant.

Finally, we note that the lift is unique. Indeed, over the dense open subset $Y_{1,{\mathrm{reg}}}$, the morphism $\rho_1$ is an isomorphism, so the action lifts uniquely. Since $X_1$ is separated, any two lifts must agree everywhere.
\end{proof}

Recall that a morphism $\rho\colon X\to Y$ is \textit{semismall} if there is no irreducible subvariety $V \subseteq X$ such that $2\dim(V)-\dim(\rho(V))>\dim(X)$. 

\begin{proposition}\label{prop:Gactionlifts}
    Let $\rho\colon X \to Y$ be a crepant partial resolution of an affine symplectic singularity $Y$, and let $G$ be a connected algebraic group acting on $Y$. Suppose that $X$ has terminal singularities.
    Then the action of $G$ lifts uniquely to an
action $G\times X\to X$ for which $\rho$ is $G$-equivariant.
\end{proposition}

\begin{proof}
Recall from \cref{lem:functorial-resolution-codim-two} that $Y_1:=Y\setminus\mathrm{Diss}(Y)$ is the complement of the dissident locus, and $X_1:=\rho^{-1}(Y_1)$. 
The open subset $Y_1$ is $G$-stable and the action of $G$ on $Y_1$ lifts uniquely to
the action on $X_1$ for which $ \rho_1:=\rho|_{X_1} \colon X_1 \to Y_1$ is $G$-equivariant. Moreover, $\mathrm{codim}_Y(Y\setminus Y_1)\geq 4$. We first observe that
\begin{equation} \label{eqn:codim_of_preimage_of_dissident_locus}
    \mathrm{codim}_X(X\setminus X_1)\geq2.
\end{equation}
Indeed, since $\rho$ is crepant, it is semismall by \cite[Proposition 3.5]{BertiniGrossiMauriMazzon2025}. Let $E$ be an irreducible component of
$X\setminus X_1$ and set
$ Z:=\overline{\rho(E)}.$
We know that
    $\mathrm{codim}_Y Z\geq 4.$
Then \eqref{eqn:codim_of_preimage_of_dissident_locus} follows from the definition of semismallness. 

Since crepant partial resolutions are, by definition, projective we can choose a $\rho$-ample line bundle $\mathcal L$ on $X$ and set $\mathcal L_1:=\mathcal L|_{X_1}$. 
Since $G$ is connected and reductive, the Sumihiro linearization theorem (see \cite[Prop.~2.4]{KKV} and citations therein) implies that some power $m$ of $\mathcal L_1$ admits a $G$-linearization. Replacing $\mathcal{L}$ with $\mathcal{L}^m$ we will assume that $\mathcal L_1$ itself is $G$-linearized. The $G$-linearization of $\mathcal L_1$ induces a $G$-action on the graded ring
   $ R(\mathcal{L}_1)  := \bigoplus_{n\geq0} \Gamma (X_1, \mathcal{L}_1^{\otimes n})$
by acting on each summand. 
This action is compatible with multiplication of sections, making
$R(\mathcal{L}_1)$ a graded $G$-algebra.

We now use normality of $X$ and
\eqref{eqn:codim_of_preimage_of_dissident_locus}. Denote by $j$ the open embedding
$
    j\colon X_1\hookrightarrow X
$. Since $\mathcal L^{\otimes n}$ is locally
free, hence reflexive, and the complement of $X_1$ has codimension at
least two, the usual extension property for reflexive sheaves on a
normal variety gives
\begin{equation*}
    \mathcal L^{\otimes n}
    \iso
    j_*\bigl(\mathcal L_1^{\otimes n}\bigr)
    \qquad (n\geq0).
\end{equation*}
Taking global sections provides canonical isomorphisms
\begin{equation*}
    \Gamma(X,\mathcal L^{\otimes n})
    \iso
    \Gamma(X_1,\mathcal L_1^{\otimes n})
    \qquad (n\geq0).
\end{equation*}
These isomorphisms are compatible with multiplication giving rise to an isomorphism of section rings
\begin{equation} \label{eqn:section_ring_on_all_X}
    R(\mathcal L) := \bigoplus_{n\geq0} \Gamma(X,\mathcal L^{\otimes n}) \osi
    \bigoplus_{n\geq0} \Gamma (X_1, \mathcal{L}_1^{\otimes n}) = R(\mathcal{L}_1).  
\end{equation}
This makes $R(\mathcal L)$ into a $G$-algebra. Since $\mathcal{L}$ is ample, there is an isomorphism $X \iso \mathrm{Proj} \, R(\mathcal{L})$ \cite[Lemma~0C6J]{stacks} such that the structure morphism
   $ \mathrm{Proj} \, R(\mathcal{L}) \to \mathrm{Spec} \, R(\mathcal{L})_0$
corresponds to $\rho\colon X\to Y$. Therefore, this endows $X$ with a $G$-action such that $\rho$ is $G$-equivariant.


%

\end{proof}

\begin{remark}
    The proof of \cref{prop:Gactionlifts} shows that even if $G$ is not connected, the action will lift to $X$ if there exists an ample line bundle $L$ in $\mathrm{Pic}(X)$ such that $L |_{X_1} \in \mathrm{Pic}(X_1)^G$. 
\end{remark}
\begin{remark}
    One can generalise \cref{prop:Gactionlifts} to any partial crepant resolution $\rho \colon X\to Y$. To do so, picking a $\Q$-factorial terminalization $Z \stackrel{\pi}{\to} X \to Y$ and a $\rho$-ample line bundle $L$ in $\Pic(X/Y)$. Then replacing $L$ by a multiple if necessary, we may assume that $\pi^* L$ admits a $G$-equivariant structure. Identifying 
    \[
    X \stackrel{\sim}{\rightarrow} \mathrm{Proj} \, \bigoplus_{m \ge 0} \Gamma(Z,\pi^* L^{\otimes m})
    \]
    over $Y$ gives an action of $G$ on $X$. 
\end{remark}

The following lemma describes the situation where we will look to apply \cref{thm:pistarmorifan}. 

\begin{lemma}\label{lem:symplectic_pic_isomorphism_from_lifting}
    If $Y$ is a conic symplectic singularity and $G$ a connected algebraic group acting on $Y$ then the action of any finite subgroup $\Gamma$ of $G$ lifts to any $\Q$-factorial terminalization $X \to Y$ and the quotient map satisfies  
    \[
    \pi^{*} \colon \operatorname{Pic}((X/\Gamma) / (Y/\Gamma))_{\mathbb{Q}} \iso \operatorname{Pic}(X/Y)_{\mathbb{Q}}.
    \]
\end{lemma}

\begin{proof}
By \cref{prop:Gactionlifts}, the action of $G$ lifts to $X$ making $X \to Y$ equivariant. Since $G$ connected and $X$ is normal, the isomorphism follows from \cref{lem:pullbackgroupiso}(iii).
\end{proof}

\begin{corollary} \label{cor:bijection_between_qterm_of_variety_and_quotient}
   Suppose conditions of \cref{lem:symplectic_pic_isomorphism_from_lifting} hold.    If $\Gamma$ acts freely on an open subset $U$ of $X$ with $\mathrm{codim}_X (X \setminus U) \ge 4$ then 
 \[
 (X_i \to Y) \mapsto (X_i / \Gamma \to Y/\Gamma)
 \]
 is a bijection between the $\Q$-factorial terminalizations of $Y$ and those of $Y/\Gamma$.  
\end{corollary}

\begin{proof}
    The variety $X$ has symplectic singularities since $X \to Y$ is a crepant partial resolution. Since $X / \Gamma \to Y/\Gamma$ is Poisson, it is also crepant by \cref{prop:Poissomorphismsympsing}; one can also see directly that the canonical bundle is trivial on both $X / \Gamma$ and $Y/\Gamma$. Therefore, by \cref{lem:finitequotientsing} the assumption that $\Gamma$ acts freely on an open subset $U$ of $X$ with $\mathrm{codim}_X (X \setminus U) \ge 4$ implies that the singular locus of $X/\Gamma$ has codimension at least four and hence is a $\Q$-factorial terminalization.  In this case, Lemma~\ref{lem:Qfactorialterminaliffsmallmodification} says that the small $\Q$-factorial modifications of $X/\Gamma \to Y/\Gamma$ are precisely the $\Q$-factorial terminalizations of $Y/\Gamma$. 
    Therefore, the claim follows from \cref{thm:pistarmorifan} (iv). 
\end{proof}

\begin{remark}
Notice that if $Y$ admits a symplectic resolution then by \cite[Corollary 31]{NamikawaPoissondeformations} every $\Q$-factorial terminalization of $Y$ is also smooth. In this case \cref{cor:bijection_between_qterm_of_variety_and_quotient} identifies symplectic resolutions of $Y$ with $\Q$-factorial terminalizations of $Y/\Gamma$.
\end{remark}

  \section{Non-minimal nilpotent orbits} \label{sec:6.Main-theorem}
  In this section we consider the action of a finite group $\Gamma$ by inner automorphisms on the closure of a nilpotent orbit $\mc{O}$ in $\mf{sl}_n$.  The group $G := \GL_n$ acts on $\overline{\mc{O}}_{\lambda}$ via the adjoint action on $\mf{sl}_n$. We say that $\Gamma$ is a group acting by inner automorphisms of $\overline{\mc{O}}$ if there is a group homomorphism $\acthom: \Gamma \to G$ such that the action of $\Gamma$ on $\mc{O}$ comes via this homomorphism. Recall  that all nilpotent orbit closures in type $A$ are normal \cite{KraftProcesi1979} and that they are labelled by the set $\{ \lambda \}$ of partitions  of $n$ \cite{CollingwoodMcGovern}. We denote by ${\mc{O}}_{\lambda}$ the corresponding orbit. The orbit closure $\overline{\mc{O}}_{\lambda}$ has symplectic singularities and all projective symplectic resolutions of $\overline{\mc{O}}_{\lambda}$ are of the form of the cotangent bundle $T^*(G/P_{\lambda}) \to \overline{\mc{O}}_{\lambda}$ \cite{FU_nilpotent_orbits_III}, where $P_{\lambda}$ is a parabolic subgroup of $G$, corresponding to $\lambda$. The group $G$ acts on $T^*(G/P_{\lambda})$ making $T^*(G/P_{\lambda}) \to \overline{\mc{O}}_{\lambda}$ equivariant. Notice that conjugacy classes of $P_{\lambda}$ are in one-to-one correspondence with  compositions of $n$ corresponding to $\lambda^{t}$. Let $T \subset G$ be the maximal torus of diagonal matrices and assume that $P_{\lambda} \supset T$ is a standard parabolic subgroup of $G$ which contains the Borel subgroup of upper-triangular matrices.

Notice that the action of $\Gamma$ on the lift $T^*(G/P_{\lambda})$ is never free. Indeed, every element of $\Gamma$ has finite order, thus is diagonalisable.  We can find a fixed partial flag of the form corresponding to $G/P_{\lb}$ and take zero covector at this point. By \cref{lem:finitequotientsing} the varieties $T^*(G/P_{\lambda})/\Gamma$ are singular.
The main theorem of this section is:
\begin{thm} \label{thm:non-minimal_quotients_are_q-factorial_and_terminal}
    All quotients $T^*(G/P_{\lambda})/ \Gamma$ are $\Q$-factorial and all of them, except possibly for the minimal nilpotent orbit quotient (case $\lambda = (2, 1^{n-2})$), are terminal. 
\end{thm}

%
Fix a projective symplectic resolution $X = T^* (G/P) \to \overline{\mc{O}} = Y$ and let $\pi \colon X \to X / \Gamma$. \cref{lem:symplectic_pic_isomorphism_from_lifting} then provides us with the isomorphism:
  \begin{equation*}
    \pi^{*} \colon \operatorname{Pic}((X/\Gamma) / (Y/\Gamma))_{\mathbb{Q}} \;\xrightarrow{\;\sim\;}\; \operatorname{Pic}(X/Y)_{\mathbb{Q}}.
    \end{equation*}
 Note that $ \operatorname{Pic}(\overline{\mc{O}}/\Gamma) = \operatorname{Pic}(\overline{\mc{O}}) = 0$. Applied in this setting, \cref{thm:pistarmorifan} yields the following:

\begin{corollary}\label{cor:pistarmorifannip}
   Assume that $\mc{O} \neq \mc{O}_{\mr{min}}$. The isomorphism $\pi^{*} \colon \operatorname{Pic}(X/\Gamma)_{\mathbb{Q}}  \iso  \operatorname{Pic}(X)_{\mathbb{Q}}$ satisfies $\pi^* \mathrm{Nef}((X/\Gamma)/ (\overline{\mc{O}}/\Gamma)) = \mathrm{Nef}(X/\overline{\mc{O}})$ and identifies the Mori fan of $X/\Gamma$ over $\overline{\mc{O}}/\Gamma$ with the Mori fan of $X$ over $\overline{\mc{O}}$.
\end{corollary}
 The key part of the proof of \cref{thm:non-minimal_quotients_are_q-factorial_and_terminal} below shows that the group $\Gamma$ acts freely on an open subset with complement of codimension at least $4$. This means that \cref{cor:bijection_between_qterm_of_variety_and_quotient} applies in this situation and we obtain the following corollary. 
\begin{corollary}\label{cor:non-minimal_quotients_are_q-factorial_and_terminal}
For $\lambda \neq (2, 1^{n-2})$, the $\Q$-factorial terminalizations of $\overline{\mc{O}}_{\lambda} / \Gamma$ are precisely the quotients $T^*(G/P_{\lambda})/ \Gamma$.
\end{corollary}
Before proving \cref{thm:non-minimal_quotients_are_q-factorial_and_terminal} we would like to formulate the following useful criterion.
\begin{criterion} \label{crt:terminality_from_tangent_space}
     Suppose $\Gamma$ is a finite group acting symplectically and faithfully on a smooth symplectic variety $M$. Then $M / \Gamma$ is terminal if and only if for each $\gamma \neq e$, and for all  $x \in M^{\gamma}$
\begin{equation*}
   \codim (T_x M)^{\gamma} = \dim_{\C} T_x M - \dim_{\C} (T_x M)^{\gamma} \ge  4.
\end{equation*}
\end{criterion}
\begin{proof}
    Recall that by \cref{thm:defn_terminal} a symplectic variety is terminal if and only if its singular locus has codimension at least $4$. Therefore, we need to compute the codimension of the singular locus of $M / \Gamma$. By \cref{lem:finitequotientsing}, the singular locus of $M / \Gamma$ is the image of the points $M^{\mr{fix}}$ in $M$ with non-trivial stabilizer. Since the quotient map is finite, the codimension of the singular locus equals the codimension of the set  $M^{\mr{fix}}$ in $M$.  

Since $\Gamma$ is finite, $M^{\mr{fix}} = \bigcup_{\gamma \in \Gamma \setminus \{ 1\} } M^{\gamma}$ and hence for each $x \in M^{\mr{fix}}$, 
\[
\dim_x M^{\mr{fix}} = \max_{\gamma} \dim_{\C} (T_x M)^{\gamma},
\]
where the max is taken over all $\gamma \in \Gamma \setminus \{ e \}$ that fix $x$.
\end{proof}

\begin{proof}[Proof of \cref{thm:non-minimal_quotients_are_q-factorial_and_terminal}]
Let $M = T^*(G/P_{\lambda})$ and $Y = M / \Gamma$. The quotient $Y$ is $\Q$-factorial by \cref{lem:Qfactorialquotient}. To show that $Y$ is terminal we would like to apply \cref{crt:terminality_from_tangent_space}.  
We thus study fixed loci of elements of our finite group $\Gamma$ acting on $T^*\left ( G/P_{\lambda} \right )$ and look for the corresponding codimensions. 
To do so we need first to describe the corresponding tangent spaces.

 Let $p := (gP_{\lambda},\alpha)  \in T^*\left ( G/P_{\lambda} \right )$ be a fixed point of $\gamma$ where $g \in G$ and $\alpha \in T^*_{gP_{\lambda}}(G/P_{\lambda})$. In particular, $\gamma \in G_p$, the stabilizer in the group $G$ of the point $p$. Let $\cdot$ denote the conjugation action.
   It is a standard fact that $T_{gP_{\lambda}}(G/P_{\lambda}) \simeq \mathfrak{g}/g \cdot \mathfrak{p}_{\lambda}$ is the adjoint representation of $G_p \subset G_{gP_{\lambda}}$
 and $T_{gP_{\lambda}}^*(G/P_{\lambda}) \simeq (\mathfrak{g}/g \cdot \mathfrak{p}_{\lambda})^* \simeq (g \cdot\lie{p}_\lambda)^{\bot}$ is the dual coadjoint representation.
  In particular for these two vector spaces the dimension (and, respectively, codimension) of subspaces of fixed points under the action of $\gamma$ is the same.
  These two spaces fit into the (split) exact sequence of $G_p$-modules  
 \begin{equation*}
    0 \rightarrow T_{gP_{\lambda}}^*(G/P_{\lambda}) \rightarrow T_p \left( T^*\left ( G/P_{\lambda} \right ) \right ) \rightarrow T_{gP_{\lambda}}(G/P_{\lambda}) \rightarrow 0 .
\end{equation*}
  This sequence stays exact after taking $\gamma$-invariants
     \begin{equation*}
    0 \rightarrow \left (T_{gP_{\lambda}}^*(G/P_{\lambda}) \right ) ^{\gamma} \rightarrow \left ( T_p \left( T^*\left ( G/P_{\lambda} \right ) \right ) \right )^{\gamma} \rightarrow \left ( T_{gP_{\lambda}}(G/P_{\lambda})\right )^{\gamma} \rightarrow 0 .
\end{equation*}
  This implies that we have the \say{doubling dimension} property 
   \begin{equation} \label{eq:cotangent}
        \dim T_p \left( T^*\left ( G/P_{\lambda} \right ) \right )^{\gamma} = \dim  \, \left( (\lie{g}/ g \cdot \mf{p}_{\lambda})^{\gamma} \oplus \left(\left ( \lie{g}/ g \cdot \mf{p}_{\lambda} \right )^*\right)^{\gamma} \right) = 2 \dim (\lie{g}/  g \cdot \mf{p}_{\lambda})^{\gamma}.
    \end{equation}
and can restrict ourselves to the action of $\gamma$ on $\mathfrak{g}/g \cdot \mathfrak{p}_{\lambda}$.

In our computations, we can always use the $G$-action to move $p \in T^*\left ( G/P_{\lambda} \right )$ to $(P_{\lambda},v)$ for some $v$. 
Note that after that move $\gamma$ acts via $g \gamma g^{-1}$ for the appropriate $g \in G$. We will call this element $\gamma$ as the remaining proof only uses the finiteness of its order. 

For $p=(P_{\lambda},v)$ $\gamma(p) = p$ implies that $\gamma \in P_{\lambda}$ since the normaliser $N_G(P_{\lambda}) = P_{\lambda}$. Then \eqref{eq:cotangent} says that $\dim T_p \left( T^*\left ( G/P_{\lambda} \right ) \right )^{\gamma} = 2 \dim (\lie{g}/  \mf{p}_{\lambda})^{\gamma}$ and it suffices to prove that $\dim (\lie{g}/ \mf{p}_{\lambda})^{\gamma} \le \dim \lie{g}/ \mf{p}_{\lambda} - 2$ for all $\lambda \neq (2,1^{n-2}), (1^n)$.

Recall that the partial flag in $V=\C^n $ corresponding to $P_{\lb}$ is a collection of vector spaces $0 \subset V_1 \subset \ldots \subset V_k \subset V $ with $\on{dim} V_i = \sum_{\sigma(j) \leq i} \lb_{j}^{t}$, such that $g(V_i) \subseteq V_i $ for all $g \in P_{\lb}$. Here $\sigma$ is the permutation of $\lambda_i$ that indexes the chosen parabolic.  
 The following is immediate.
    \begin{claim}
    \label{clm:basis-for-parabolic}
        There is a basis $e_1, \ldots, e_n$ of $V$ consisting of eigenvectors for $\gamma$ such that every subspace in the partial flag stabilized by $P_{\lambda}$ has a basis consisting of a subset of these eigenvectors. 
    \end{claim} 
Indeed, one constructs such basis by induction by taking a basis of $\gamma$-eigenvectors in $V_1$ and completing it step-by-step to a complete basis of $V$.
With respect to this basis, $\gamma$ belongs to the maximal torus of diagonal matrices, which is contained in $P_{\lambda}$. We denote by $\xi_i$ the corresponding eigenvalues $\gamma \cdot e_i = \xi_i e_i$.  If $e_1, \dots, e_n$ is the standard basis of $V$ then $P_{\lambda}$ is block upper-triangular. 

Let $e_1^*, \dots, e_n^*$ be the dual basis in $V^*$. A basis in $\lie{g}/ \mf{p}_{\lambda}$ is given by $e_i \otimes e_j^*$, where the entry $(i,j)$ is outside the upper triangular part corresponding to $P_{\lb}$. The element $e_i \otimes e_j^* \in \lie{g}/ \mf{p}_{\lambda}$ is an eigenvector for $\gamma \in \Gamma$ with eigenvalue $\xi_i\xi_j^{-1}$.
    
    
    \begin{example}
        For $\lambda = (2,2,1)$ we have $\lambda^t = (3,2)$, and the tangent space $\lie{g}/\lie{p_{\lambda}}$ can be realised by matrices as in the picture below,  where every $*$ denotes an arbitrary number, with the red part corresponding to $\lie{p_{\lambda}}$. 
    \end{example}

\begin{center}

 \begin{tikzpicture}
 \matrix [matrix of math nodes,left delimiter=(,right delimiter=)] (m)
{
0 & 0 & 0& 0 &0 \\
0 & 0 & 0& 0 &0 \\
0 & 0 & 0& 0 &0 \\
* & * & *& 0 &0 \\
* & * & *& 0 &0\\
};

\draw[
fill=red,
opacity=0.5] (m-3-1.south west) rectangle (m-1-3.north east);
\draw[
fill=red, 
opacity=0.5] (m-5-4.south west) rectangle (m-4-5.north east);
\draw[
fill=red,
opacity=0.5] (m-3-4.south west) rectangle (m-1-5.north east);
\end{tikzpicture}
\end{center}

 \cref{crt:terminality_from_tangent_space} explains that the quotient $M/\Gamma$ is terminal iff for each element $\gamma \in \on {Stab}_{\Gamma}(x)$, $\gamma \ne e$, we have $4 \le \dim_{\C} T_x M - \dim_{\C} (T_x M)^{\gamma}.$ For the quotient  $T^*(G/P_{\lambda})/\Gamma$ from the above description of the tangent space the difference $\dim_{\C} T_x M - \dim_{\C} (T_x M)^{\gamma}$ translates to twice the number of elements $e_{ij}:= e_i \otimes e_j^*  \notin \mf{p}_{\lambda}$ such that $\xi_i \xi_j^{-1} \neq 1$ i.e., $\xi_i \neq \xi_j$. In particular, we will show that there always exist at least two such pairs - we will call this the \textit{terminality condition} for $\gamma$ at $x$. The following is a reformulation of \cref{crt:terminality_from_tangent_space} applied to $T^*(G/P_{\lambda})/\Gamma$.
  \begin{claim}
      If the terminality condition is satisfied for each element $\gamma \in \Gamma$ ($\gamma \ne e$) and each point in fixed point locus of $\gamma$ then the quotient variety $M / \Gamma$ is terminal.
  \end{claim}

We will reduce the problem to the extreme case of maximal parabolic subalgebras and check this condition there. Indeed, if the condition is satisfied for some $\lie{p}_{max} \supset \lie{p}$ then it is satisfied for $\lie{p}$ as well, because
\begin{equation*}
    \on{dim} (G/P)^{\gamma}  \leq \on{dim} (G/P_{max})^{\gamma}+ \on{dim} (P_{max}/P), 
\end{equation*}
    which implies that 
    $  \on{codim}_{G/P}(G/P)^{\gamma} \geq \on{codim}_{G/P_{max}}(G/P_{max})^{\gamma}.
$ 

\begin{lemma}
    If $\lie{p} \subset \lie{sl}_n$ is a parabolic Lie subalgebra which is not maximal, does not correspond to a composition $(1, k, 1)$ and $n>3$ then $\lie{p}$ lies in a maximal parabolic $\lie{p}_{max}$ that is not a $\mf{p}_{\lambda}$ for $\lambda = (2, 1^{n-2})$.
\end{lemma}
\begin{proof}
    For $\lie{sl}_n$  each parabolic $\lie{p}$ after conjugation has block upper-triangular form, corresponding to some composition $\lambda_{\lie{p}}= (n_1, \ldots, n_i)$. Suppose that $i>2$. We can then choose a smallest block $n_j$ with a smallest neighbouring block $n_{j \pm 1}$ and consider the composition $\lambda'$ which is obtained from $\lambda_{\lie{p}}$ by replacing these two parts with their sum. The composition $\lambda' = \lambda_{\lie{p'}}$ corresponds to a parabolic subalgebra ${\lie{p'}}\supset {\lie{p}}$.
    We repeat this procedure until we only have two blocks. Unless $n=3$ or $\lambda_{\lie{p}} = (1, k, 1)$  we will always end up with a composition that is not $(n-1, 1)$ or $(1, n-1)$.
\end{proof}

\begin{claim} \label{claim:2-2_terminality}
   Consider a maximal parabolic, corresponding to a composition $(n_1, n_2)$ with $n_1, n_2 \ge 2$. Suppose a (finite) group element $\gamma\in \Gamma$ acts non-trivially on $T^*(G/P_{max})$ with a fixed point $p =(gP_{max},v)$. Then the terminality condition holds for $\gamma$ acting on the tangent space 
      $\lie{g}/ g \cdot \lie{p}_{max}\oplus (\lie{g}/g \cdot \lie{p}_{max})^*$
      at $p$.
\end{claim}
\begin{proof}
The action of $\gamma$  on $\C^n$ is diagonalisable. Since the inner automorphism induced by $\gamma$ is non-trivial, $\gamma$ does not lie in the center of $\SL_n$ and therefore has at least two distinct eigenvalues.

Consider an eigenbasis for $\gamma$ and $P_\lambda = g \cdot P_{max}$ defined in \cref{clm:basis-for-parabolic}.
Let $e_k, e_l$ be eigenvectors of $\gamma$ with distinct eigenvalues. 
   If they belong to the same block we can consider   a four-dimensional space spanned by $e_k, e_l, e_i, e_j$ where $e_i$ and $e_j$ are any two eigenvectors from another block ($i,j > k,l$). 
   If the eigenvalue $\xi_i$ is different from both $\xi_k$ and $\xi_l$ then $\gamma$ acts non-trivially on $e_{ik}$ and $e_{il}$. Otherwise we can assume that $\xi_i=\xi_k$. If $\xi_j=\xi_k$ then $\gamma$ acts non-trivially on $e_{il}$ and $e_{jl}$. If $\xi_j \ne \xi_k$ then $\gamma$ acts non-trivially on $e_{il}$ and $e_{jk}$.

If two eigenvectors with different eigenvalues always belong to different blocks this means that each block has a distinguished eigenvalue. We can then take $e_{1},e_{2}$ from the first block and $e_n$ from the second one. Then $\gamma$ acts non-trivially on $e_{n1}$ and $e_{n2}$ and thus the terminality condition holds in this case as well.

\end{proof}

     \cref{claim:2-2_terminality} does not include $n=3$ case and the case $\lambda_{\lie{p}} = (1, k, 1)$. 

    The quotient is terminal for the case $\lambda_{\lie{p}} = (1, k, 1)$. Indeed, take $\gamma$ inducing a non-identity automorphism of $G/P_{\lambda}$ and consider the action on $\lie{g}/\lie{p}_{\lambda}$. Suppose that $\xi_1\ne \xi_n$. Then $\gamma$ acts non-trivially on $e_{n1}$ and either $e_{21}$ or $e_{n2}$. If $\xi_1 = \xi_n$ then there exists $\xi_i\ne \xi_1$ and $\gamma$ acts non-trivially on $e_{i1}$ and $e_{ni}$.
The case $n=3$ follows by taking $k=1$.

    This completes the proof of \cref{thm:non-minimal_quotients_are_q-factorial_and_terminal}.
\end{proof}
\begin{remark}
    Notice that if all non-identity elements $\gamma \in \Gamma$ have at least three different eigenvalues or two eigenvalues with at least two eigenvectors each in the vector space $V$ then the proof of \cref{thm:non-minimal_quotients_are_q-factorial_and_terminal} works for the minimal nilpotent orbit and the corresponding quotient is terminal. For example, this means that a quotient by the alternating group $A_n$ will be terminal. 
\end{remark}


\printbibliography

\end{document}